\documentclass[a4paper,10pt]{article}
\usepackage[margin=1.25in]{geometry}
\usepackage[]{algorithm2e}
\usepackage{amsthm}
\usepackage{caption}                         

\usepackage[utf8]{inputenc}
\usepackage{amssymb}
\usepackage{cleveref}
\usepackage{verbatim}
\usepackage{graphicx}
\usepackage{epstopdf}
\usepackage{indentfirst}                     
\usepackage[english]{babel}
\usepackage{subfig}                          
\usepackage{chngpage, calc}                  
\usepackage[font=small]{quoting}             
\usepackage[autostyle]{csquotes}             
\usepackage{hyperref}                        
\usepackage{cleveref}               
\usepackage{amsmath}
\usepackage{mathtools}
\usepackage{bookmark}
\usepackage{mdframed}
\usepackage{xcolor}
\usepackage{enumerate}
\usepackage{algorithmic}
\Crefname{ALC@unique}{Line}{Lines}

\usepackage[style=numeric,hyperref,backend=bibtex, maxcitenames=5, maxbibnames=9]{biblatex}
\bibliography{phase-field}

\numberwithin{equation}{section}

\newcommand{\Tau}{\mathcal{T}}

\definecolor{lightblue}{rgb}{0.68, 0.85, 0.9}
\definecolor{pastelblue}{rgb}{0.68, 0.78, 0.81}
\definecolor{palegreen}{rgb}{0.6, 0.98, 0.6}
\definecolor{pastelred}{rgb}{1.0, 0.41, 0.38}

\mdfdefinestyle{exampledefault}{
rightline=false,bottomline=false,
topline=false,leftline=false,
backgroundcolor=white
}
\mdfdefinestyle{problemdefault}{
rightline=false,bottomline=false,
topline=false,leftline=false,
backgroundcolor=white
}
\mdfdefinestyle{badthingdefault}{
rightline=false,bottomline=false,
topline=false,leftline=false,
backgroundcolor=white
}

\newcommand{\ue}{u_{\varepsilon}}

\newcommand{\Je}{J_{\varepsilon}}

\newcommand {\half}{ \frac{1}{2} }

\newcommand{\N}{\mathbb{N}} 
\newcommand{\R}{\mathbb{R}}

\newcommand{\norm}[2]{{\left\| #2 \right\|}_{#1}}

\newcommand{\diverg}{\mathrm{div}}

\DeclareMathOperator*{\argmin}{arg\,min}

\newcommand{\normal}{\nu}

\theoremstyle{definition} 

\newtheorem{remark}{Remark}[section]

\theoremstyle{plain}

\newtheorem*{theorem*}{Theorem}
\newtheorem{proposition}{Proposition}[section]
\newtheorem*{proposition*}{Proposition}
\newtheorem{lemma}{Lemma}[section]

\renewbibmacro{in:}{}
\DeclareFieldFormat[article]{citetitle}{#1}
\DeclareFieldFormat[article]{title}{#1}
\DeclareFieldFormat[article]{pages}{#1}

\DeclareFieldFormat[inproceedings]{citetitle}{#1}
\DeclareFieldFormat[inproceedings]{title}{#1}

\hypersetup{
   colorlinks=true,
	linkcolor = {black},
	citecolor = {black},
}

\colorlet{linkequation}{blue}

\newcommand*{\SavedEqref}{}
\let\SavedEqref\eqref
\renewcommand*{\eqref}[1]{%
  \begingroup
    \hypersetup{
      linkcolor=linkequation,
      linkbordercolor=linkequation,
    }%
    \SavedEqref{#1}%
  \endgroup
}

\graphicspath{{img/}} 

\usepackage{tikz}
\newcommand*\circled[1]{\tikz[baseline=(char.base)]{
    \node[shape=circle,draw,inner sep=2pt] (char) {#1};}}

\title{A phase-field approach for the interface reconstruction in a nonlinear elliptic problem arising from cardiac electrophysiology}
\author{
  Elena Beretta\footnotemark[1], \
  Luca Ratti\footnotemark[1], \ 
	Marco Verani\footnotemark[1]
}
\date{ }

\begin{document}

\maketitle
\renewcommand{\thefootnote}{\fnsymbol{footnote}}
\footnotetext[1]{Dipartimento di Matematica, Politecnico di Milano, P.za Leonardo da Vinci 32, 20133 Milano, Italy, \texttt{ \{elena.beretta,luca.ratti,marco.verani\}@polimi.it} }

\begin{abstract}
In this work we tackle the reconstruction of discontinuous coefficients in a semilinear elliptic equation from the knowledge of the solution on the boundary of the domain, an inverse problem motivated by  biological application in cardiac electrophysiology.
\par
We formulate a constraint minimization problem involving a quadratic mismatch functional enhanced with a regularization term which penalizes the perimeter of the inclusion to be identified. We introduce a phase-field relaxation of the problem, replacing the perimeter term with a Ginzburg-Landau-type energy. We prove the $\Gamma$-convergence of the relaxed functional to the original one (which implies the convergence of the minimizers), we compute the optimality conditions of the phase-field problem and define a reconstruction algorithm based on the use of the Frèchet derivative of the functional. After introducing a discrete version of the problem we implement an iterative algorithm and prove convergence properties. Several numerical results are reported, assessing the effectiveness and the robustness of the algorihtm in identifying arbitrarily-shaped inclusions.
\par
Finally, we compare our approach to a shape derivative based technique, both from a theoretical point of view (computing the sharp interface limit of the optimality conditions) and from a numerical one. 
\end{abstract}

\section{Introduction}
We consider the following Neumann problem, defined over $\Omega \subset \mathbb{R}^2$:
\begin{equation}
	\left\{
	\begin{aligned}
		-\diverg(\tilde{k}(x) \nabla y) + \chi_{\Omega \setminus \omega}y^3 &= f \qquad \text{in } \Omega	\\
		\partial_{\normal} y &= 0 \qquad \text{on } \partial \Omega,
	\end{aligned}
	\right.
	\label{eq:probclassic}
\end{equation}
where $\chi_{\Omega \setminus \omega}$ is the indicator function of $\Omega \setminus \omega$ and
\[
	\tilde{k}(x) = 
	\left\{ 
	\begin{aligned} 
		k & \text{ if } x \in \omega \\ 
		1 & \text{ if } x \in \Omega \setminus \omega, 
	\end{aligned} 
	\right. 
	\quad k_{in} \neq k_{out},
\]
being $ 0 < k \ll 1$ and $f \in L^2(\Omega)$.

The boundary value problem \eqref{eq:probclassic} consists of a semilinear diffusion-reaction equation with discontinuous coefficients across the interface of an inclusion $\omega \subset \Omega$, in which the conducting properties are different from the background medium. Our goal is the determination of the inclusion from the knowledge of the value of $y$ on the boundary $\partial \Omega$, i.e., given the measured data $y_{meas}$ on the boundary $\partial \Omega$, to find $\omega \subset \Omega$ such that the corresponding solution $y$ of \eqref{eq:probclassic} satisfies
	\begin{equation} 
		y|_{\partial \Omega} = y_{meas}. 
		\label{eq:invintro}
	\end{equation}
	Since at the state of the art very few works tackle similar inverse problems in a nonlinear context, the reconstruction problem to which this work is devoted is particularly interesting from both an analytic and a numerical standpoint. 
	
	The direct problem can be related to a meaningful application arising in cardiac electrophysiology, up to several simplifications. In that context (see \cite{book:sundes-lines}, \cite{book:pavarino}), the solution $y$ represents the electric transmembrane potential in the heart tissue, the coefficient $\tilde k$ is the tissue conductivity and the nonlinear reaction term encodes a ionic transmembrane current. An inclusion $\omega$ models the presence of an ischemia, which causes a substantial alteration in the conductivity properties of the tissue.

 The objective of our work, in the long run, is the identification of ischemic regions through a set of measurements of the electric potential acquired on the surface of the myocardium.  
 Indeed, a map of the potential on the boundary of internal heart cavities can be acquired by means of non-contact electrodes carried by a catheter inside a heart cavity; this is the procedure of the so-called intracardiac electrogram technique, which has become a possible (but invasive) inspection technique for patients showing symptoms of heart failure.
 We remark that our model is a simplified version of the more complex \textit{monodomain} model (see e.g. \cite{phd:tung}, \cite{book:sundes-lines}). The monodomain is a continuum model which describes the evolution of the transmembrane potential on the heart tissue according to the conservation law for currents and to a satisfying description of the ionic current, which entails the coupling with a system of ordinary differential equations for the concentration of chemical species. In this preliminary setting, we remove the coupling with the ionic model, adopt instead a phenomenological description of the ionic current, through the introduction of a cubic reaction term. Moreover, we consider the stationary case in presence of a source term which plays the role of the electrical stimulus.

 
\par 
Despite the simplifications, the problem we consider in this paper is a mathematical challenge itself. Indeed, here the difficulties include the nonlinearity of both the direct and the inverse problem, as well as the lack of measurements at disposal.  
\par
In fact, already the linear counterpart of the problem, obtained when the nonlinear reaction term is removed, is strictly related to the \textit{inverse conductivity problem}, also called \textit{Calder\'{o}n problem}, which has been object of several studies in the last decades. Without additional hypotheses on the geometry of the inclusion, but only assuming a sufficient degree of regularity of the interface, uniqueness from knowledge of infinitely many measurements has been proved in \cite{art:isakov_discontinuous} and logarithmic-type stability estimates have been derived in \cite{art:ales}. Finitely many measurements are sufficient to determine uniquenely and in a stable (Lipschitz) way the inclusion introducing additional information either on the shape of the inclusion or on its size, e.g. when the inclusion belongs to a specific class of domains with prescribed shape, such as discs, polygons, spheres, cylinders, polyhedra (see \cite{art:isa-po}, \cite{book:ammari-kang}, \cite{art:barcelo}) or when the volume of the inclusion is small compared to the volume of the domain (see \cite{art:fried-vog}, \cite{art:cfmv}). 
\par
Several reconstruction algorithms has been developed for the solution of the inverse conductivity problem, and it is beyond the purposes of this introduction to provide an exhaustive overview on the topic. Under the assumption that the inclusion to be reconstructed is of small size, we mention the constant current projection algorithm in \cite{art:amm-seo}, the least-squares algorithm proposed in \cite{art:cfmv}, and the linear sampling method in \cite{art:bhv} for similar problems. Although these algorithms have proved to be effective, they heavily rely on the linearity of the problem. On the contrary, it is possible to overcome the strict dependence on the linearity of the problems by aims of a variational approach, based on the constraint minimization of a quadratic misfit functional, as in \cite{art:kv1987}, \cite{art:amstutz2005crack} and \cite{art:ammari2012}. When dealing with the reconstruction of extended inclusions in the linear case, both direct and variational algorithms are available. Among the first ones, we mention \cite{art:bhv} and \cite{art:ikehata2000}; instead, from a variational standpoint, a shape-optimization approach to the minimization of the mismatch functional, with suitable regularization, is explored in \cite{art:kalvk} \cite{art:hett-run}, \cite{art:afraites} and \cite{art:abfkl}. In \cite{art:hintermuller2008} and \cite{art:cmm}, this approach is coupled with topology optimization; whereas the level set technique coupled with shape optimization technique has been applied in \cite{art:santosa} and in \cite{art:ito-kunish}, \cite{art:burger2003}, \cite{art:chan-tai}, also including a Total Variation regularization of the functional. Recently, total-variation-based schemes have been employed to solve inverse problems: along this line we mention, among the others, the Levenberg-Marquardt and Landweber algorithms in \cite{art:bb} and the augmented Lagrangian approach in \cite{art:chen1999}, \cite[Chapter 10]{book:bartels} and \cite{art:bartels2012}. Finally, the phase field approach has been explored for the linear inverse conductivity problem e.g in \cite{art:rondi2001} and recently in \cite{art:deckelnick}, but consists in a novelty for the non-linear problem considered in this paper.
\par
Concerning the reconstruction algorithm for inverse problems dealing with non-linear PDEs, we recall some works related to sensitivity analysis for semilinear elliptic problems as \cite{art:scheid}, \cite{art:amstutzNL}, although in different contexts with respect to our application. We remark that the level-set method has been implemented for the reconstruction of extended inclusion in the nonlinear problem of cardiac electrophysiology (see \cite{art:lyka} and \cite{chavez2015}), by evaluating the sensitivity of the cost functional with respect to a selected set of parameters involved in the full discretization of the shape of the inclusion.
In \cite{art:BMR} the authors, taking advantage from the results obtained in \cite{art:bcmp}, proposed a reconstruction algorithm for the nonlinear problem \eqref{eq:probclassic} based on topological optimization, where a suitable quadratic functional is minimized to detect the position of small inclusions separated from the boundary. In \cite{art:BCCMR} the results obtained in \cite{art:BMR} and \cite{art:bcmp} have been extended to the time-dependent monodomain equation under the same assumptions. Clearly, this type of assumptions on the unknown inclusions are quite restrictive particulary for the application we have in mind.  
In this paper we propose a reconstruction algorithm of conductivity inclusions of arbitrary shape and position by relying on the minimization of a suitable boundary misfit functional, enhanced with a perimeter penalization term, and, following the approach in \cite{art:deckelnick}, by introducing a relaxed functional obtained by using a suitable phase field approximation, where the discontinuity interface of the inclusion is replaced by a diffuse interface with small thickness expressed in terms of a positive relaxation parameter $\varepsilon$ and the perimeter functional is replaced by the Ginzburg-Landau energy. 
\par
The outline of the paper is as follows: in Section \ref{sec:regul} we introduce and motivate the Total Variation regularization for the optimization problem. In Section \ref{sec:relax}, after introducing the phase-field regularization of the problem and discussing its well-posedness, we show $\Gamma$-convergence of the relaxed functional to the original one as the relaxation parameter approaches zero. We furthermore derive necessary optimality conditions associated to the relaxed problem, exploiting the Fréchet derivative of the functional.
The computational approach proposed in Section \ref{sec:discr} is based on a finite element approximation similarly to the one introduced in \cite{art:deckelnick}.  Despite the presence of the nonlinear term in the PDE it is possible to show that the  the discretized solution converges to the solution of the phase field problem. We derive an iterative method which is shown to yield an energy decreasing sequence converging to a discrete critical point. The power of this approach is twofold: on one hand it allows to consider conductivity inclusions of arbitrary shape and position which is the case of interest for our application and on the other it leads to remarkable reconstructions as shown in the numerical experiments in Section \ref{sec:numer}. 
Finally, in Section \ref{sec:shape} we compare our technique to the shape optimization approach: after showing the optimality conditions derived for the relaxed problem converge to the ones corresponding to the sharp interface one, we show numerical results obtained by applying both the algorithms on the same benchmark cases.

\section{Minimization problem and its regularization}
\label{sec:regul}
In this section, we give a rigorous formulation both of the direct and of the inverse problem in study. The analysis of the well-posedness of the direct problem is reported in details in the Appendix, and consists in an extension of the results previously obtained in \cite{art:bcmp}. The well-posedness of the inverse problem is analysed in this section: in particular, we formulate an associated constraint minimization problem and investigate the stability of its solution under perturbation of the data, following an approach analogous to \cite[Chapter 10]{book:engl}, but setting the entire analysis in a non-reflexive Banach space, which entails further complications. The strategy adopted to overcome the instability is the introduction of a Tikhonov regularization, and the properties of the regularized problem are reported and proved in details. 
\par
We formulate the problems \eqref{eq:probclassic} and \eqref{eq:invintro} in terms of the indicator function of the inclusion, $u = \chi_{\omega}$. We assume an \textit{a priori} hypothesis on the inclusion, namely that it is a subset of $\Omega$ of finite perimeter: hence, $u$ belongs to $BV(\Omega)$, i.e. the space of the $L^1(\Omega)$ functions for which the Total Variation is finite, being
\[
TV(u) = \sup \left\{ \int_\Omega u \text{div}(\phi); \quad \phi \in C^1_0(\Omega;\R^2), \ \norm{L^\infty}{\phi}\leq 1 \right\},
\]
endowed with the norm $\norm{BV}{\cdot} = \norm{L^1}{\cdot} + TV(\cdot)$. In particular, 
\[
u \in X_{0,1}= \{ v \in BV(\Omega): v(x) \in \{0,1\} \text{ a.e. in $\Omega$ }\}.
\]
The weak formulation of the direct problem \eqref{eq:probclassic} in terms of $u$ reads: find $y$ in $H^1(\Omega)$ s.t., $\forall \varphi \in H^1(\Omega)$,
\begin{equation}
	\int_{\Omega} a(u) \nabla y \nabla \varphi + \int_{\Omega} b(u) y^3 \varphi = \int_{\Omega} f \varphi,
\label{eq:prob}
\end{equation} 
being $a(u) = 1-(1-k)u$ and $b(u) = 1-u$. 
Define $S: X_{0,1} \rightarrow H^1(\Omega)$ the \textit{solution map}: for all $u \in X_{0,1}$, $S(u) = y$ is the solution to problem \eqref{eq:prob} with indicator function $u$; the inverse problem consists in:
\begin{equation}
 \textit{ find } u \in X_{0,1} \textit{ s.t. } S(u)|_{\partial \Omega} = y_{meas}.
\label{eq:inv}
\end{equation}
As it is proved in the Appendix, in Proposition \ref{prop:wellpos}, the solution map $S$ is well defined between the spaces $BV(\Omega;[0,1])$ and $H^1(\Omega)$, thus for each $u \in X_{0,1}$ there exists a unique solution $S(u) \in H^1(\Omega)$.
\par 
We introduce the following constraint optimization problem:
\begin{equation}
	\argmin_{u \in X_{0,1}} J(u); \qquad J(u) = \half \norm{L^2(\partial \Omega)}{S(u) - y_{meas}}^2.
\label{eq:opt}
\end{equation}
It is well known that this problem is ill-posed, and in particular instability under the perturbation of the boundary data occurs.

\par
A possible way to recover well-posedness for the minimization problem in $BV(\Omega)$ is to introduce a Tikhonov regularization term in the functional to minimize, e.g. a penalization term for the perimeter of the inclusion. The regularized problem reads:
\begin{equation}
	\argmin_{u \in X_{0,1}} J_{reg}(u); \quad J_{reg}(u) = \half \norm{L^2(\partial\Omega)}{S(u) - y_{meas}}^2 + \alpha TV(u),
\label{eq:minreg}
\end{equation}
Then, we can prove the following properties:
\begin{itemize}
	\item for every $\alpha >0$ there exists at least one solution to \eqref{eq:minreg};
	\item the solutions of \eqref{eq:minreg} are stable w.r.t. perturbation of the data $y_{meas}$;
	\item if $\{\alpha_k\}$ is a sequence of penalization parameters suitably chosen, then the sequence of the corresponding minimizers $\{u_k\}$ has a subsequence converging to a minimum-norm solution of \eqref{eq:opt}.
\end{itemize}

Before proving the listed statements, it is necessary to formulate and prove a continuity result for the solution map with respect to the $L^1$ norm, which consists in an essential property for the following analysis and requires an accurate treatment due to the non-linearity of the direct problem.
\begin{proposition}
Let $f \in L^2(\Omega)$ satisfy the hypotheses in Proposition \ref{prop:geq} or \ref{prop:geq2}.	If $\{ u_n \} \subset X_{0,1}$ s.t. $u_n \xrightarrow{L^1} u \in X_{0,1}$, then $S(u_n)|_{\partial \Omega} \xrightarrow{L^2(\partial \Omega)}S(u)|_{\partial \Omega}$.
\label{prop:continuity}
\end{proposition}
\begin{proof}
Define $w_n = S(u_n) - S(u)$; then, subtracting the \eqref{eq:prob} evaluated in $u_n$ and the same one evaluated in $u$, $w_n$ is the solution of:
\begin{equation}
	\int_{\Omega} a(u_n) \nabla w_n \nabla \varphi + \int_{\Omega} b(u_n) q_n w_n \varphi =  \int_{\Omega} (1-k)(u_n - u) \nabla S(u) \nabla \varphi - \int_{\Omega} (u_n - u) S(u)^3 \varphi,
\label{eq:diff}
\end{equation} 
where $q_n = S(u_n)^2 + S(u_n) S(u) + S(u)^2$. Considering $\varphi = w_n$ and taking advantage of the fact that $a(u_n) \geq k$ and (by simple computation) $q_n \geq \frac{3}{4}S(u)^2$, we can show that, via Cauchy-Schwarz inequality,
\[
\begin{aligned}
k \norm{L^2(\Omega)}{\nabla w_n}^2 + \frac{3}{4}\int_\Omega b(u_n) S(u)^2 w_n^2 \leq &(1-k)\norm{L^2(\Omega)}{(u_n-u)\nabla S(u)} \norm{L^2(\Omega)}{\nabla w_n} 
 \\ &+\norm{L^2(\Omega)}{(u_n-u)S(u)^3}\norm{L^2(\Omega)}{w_n}.
\end{aligned}
\]
We remark that $(u_n - u) S(u)^3 \in L^2(\Omega)$ since $S(u) \in H^1(\Omega) \subset \subset L^6(\Omega)$. Moreover,
\[
\begin{aligned}
k \norm{L^2(\Omega)}{\nabla w_n}^2 + \frac{3}{4}\int_\Omega b(u) S(u)^2 w_n^2 \leq &(1-k)\norm{L^2(\Omega)}{(u_n-u)\nabla S(u)} \norm{L^2(\Omega)}{\nabla w_n} \\&+ \norm{L^2(\Omega)}{(u_n-u)S(u)^3}\norm{L^2(\Omega)}{w_n} + \frac{3}{4}\int_\Omega (u_n-u)S(u)^2 w_n^2
\end{aligned}
\]
Thanks to Proposition \ref{prop:geq}, $\exists Q>0,$ $\Omega^*\subset\Omega$ s.t. $|\Omega^*|\neq 0$ s.t. 
\[
k \norm{L^2(\Omega)}{\nabla w_n}^2 + \frac{3}{4}Q \norm{L^2(\Omega^*)}{w_n}^2 \leq (q_1 + q_2 + q_3) \norm{H^1(\Omega)}{w_n},
\]
where $q_1 = \norm{L^2(\Omega)}{(u_n - u)\nabla S(u)}$, $q_2 = \norm{L^2(\Omega)}{(u_n - u) S(u)^3}$ and $q_1 = \frac{3}{4}\norm{L^2(\Omega)}{(u_n - u) S(u)^2}$, which implies, thanks to the Poincarè inequality in Lemma \ref{Lemma1}, 
\[
\norm{H^1(\Omega)}{w_n} \leq C (q_1 + q_2 + q_3).
\]
Consider 
\[
q_1 = \left( \int_\Omega (u_n - u)^2 |\nabla S(u)|^2 \right)^\half;
\]
since  $u_n \xrightarrow{L^1} u$, then (up to a subsequence) $u_n \rightarrow u$ pointwise almost everywhere, then also the integrand $(u_n - u)^2 |\nabla S(u)|^2$ converges to $0$. Moreover, $|u_n - u| \leq 1$, hence $\forall n$ $(u_n - u)^2 |\nabla S(u)|^2 \leq |\nabla S(u)|^2 \in L^1(\Omega)$, and thanks to Lebesgue convergence theorem, we conclude that $q_1 \rightarrow 0$. Analogously, $q_2 \rightarrow 0$, $q_3 \rightarrow 0$ and eventually $\norm{H^1(\Omega)}{w_n} \rightarrow 0$, i.e. $S(u_n) \xrightarrow{H^1} S(u)$. Thanks to the trace inequality, we can assess that also $S(u_n)|_{\partial \Omega} \xrightarrow{L^2(\partial \Omega)} S(u)|_{\partial \Omega}$.
\end{proof}

It is now possible to verify the expected properties of the regularized optimization problem.
\begin{proposition}
	For every $\alpha >0$ there exists a solution of \eqref{eq:minreg}
\label{prop:exist}
\end{proposition}
\begin{proof}
	Let $\{u_n\}$ be a minimizing sequence: then $\{S(u_n)|_{\partial \Omega}\}$ is bounded in $L^2(\partial \Omega)$ and $\{u_n\}$ is bounded in $BV(\Omega)$ (since $\{TV(u_n)\}$ is bounded and $\norm{L^1(\Omega)}{u_n}\leq |\Omega|$ for all $u_n \in X_{0,1}$). Thanks to the result of compactness for the $BV$ space (see \cite{book:ambrosiofuscopallara}, Theorem 3.23), there exists a subsequence $u_{n_k}$ weakly converging to an element $\bar{u} \in BV(\Omega)$. Moreover, being $\mathcal{D}(S)$ weakly closed, $\bar{u} \in \mathcal{D}(S)$. Since the weak $BV-$convergence implies the $L^1-$convergence, thanks to Proposition \ref{prop:continuity} we can assess that $S(u_{n_k}) \rightarrow S(\bar{u})$ in $H^1(\Omega)$ and in $L^2(\partial \Omega)$. Eventually, this proves that $\norm{L^2(\partial\Omega)}{S(u_{n_k})-y_{meas}} \rightarrow \norm{L^2(\partial\Omega)}{S(u)-y_{meas}}$. Anologously, by semi-continuity of the total variation with respect to the weak convergence in BV, $TV(\bar{u}) \leq \liminf_{k} TV(u_{n_k})$, and  it is possible to conclude that
\[
	\half \norm{L^2(\partial\Omega)}{S(u)-y_{meas}}^2 + \alpha TV(u) \leq \liminf_k (\half \norm{L^2(\partial\Omega)}{S(u_{n_k})-y_{meas}}^2 + \alpha TV(u_{n_k})),
\]
thus $u$ is a minimum of the functional. 
\end{proof}
Even if the existence of the solution is ensured by the previous result, uniqueness cannot be guaranteed since the functional is neither linear nor convex (in general).
We now investigate the stability of the minimizer of the regularized cost functional with respect to small perturbations of the boundary data. We point out that, due to the non-reflexivity of the Banach space $BV$, it is not possible to formulate a stability result with respect to the strong $BV$ convergence; nevertheless, we can perform the analysis with respect to the \textit{intermediate convergence} of $BV$ functions. A sequence $\{u_n\} \subset BV(\Omega)$ tends to $u \in BV(\Omega)$ in the sense of the intermediate convergence iff $ u_n\xrightarrow{L^1} \bar{u}$ and $TV(u_n)\rightarrow TV(\bar{u})$.

\begin{proposition}
	Fix $\alpha >0$ and consider a sequence $\{y_k\}\subset L^2(\partial \Omega)$ such that $y_k \rightarrow y_{meas}$ in $L^2(\partial \Omega)$. Consider the sequence $\{u_k\}$, where $u_k$ is a solution of \eqref{eq:minreg} with datum $y_k$. Then there exists a subsequence $\{u_{k_n}\}$ which converges to a minimizer $\bar{u}$ of \eqref{eq:minreg} with datum $y_{meas}$ in the sense of the intermediate convergence.
\label{prop:stability}
\end{proposition}
\begin{proof}
	For every $u_k$, we have that
	$$
	\half \norm{L^2(\partial \Omega)}{S(u_k)-y_k}^2 + \alpha TV(u_k) \leq \half \norm{L^2(\partial \Omega)}{S(u)-y_k}^2 + \alpha TV(u) \quad \forall u \in \mathcal{D}(S).
	$$
Hence,  $\{\norm{L^2(\partial \Omega)}{S(u_k)}\}$ and $\{TV(u_k)\}$ (and therefore $\{\norm{BV(\Omega)}{u_k}\}$) are bounded, and there exists a subsequence $\{u_{k_n}\}$ such that both $u_{k_n} \rightharpoonup \bar{u}$ in $BV(\Omega)$ and $S(u_{k_n}) \rightarrow S(\bar{u})$ in $L^2(\partial \Omega)$. Thanks to the continuity of the map $S$ with respect to the convergence (in $L^1$) of $u_{k_n}$ and to the weak lower semi-continuity of the $BV(\Omega)$ norm,
\begin{equation}
\begin{aligned}
	\half \norm{L^2(\partial \Omega)}{S(\bar{u})-y_{meas}}^2 + \alpha TV(\bar{u}) &\leq \lim \inf_n \left(\half \norm{L^2(\partial \Omega)}{S(u_{k_n})-y_{k_n}}^2 + \alpha TV(u_{k_n}) \right) \\
& \leq \lim_{n} \left(\half \norm{L^2(\partial \Omega)}{S(u)-y_{k_n}}^2 + \alpha TV(u) \right) \quad \forall u \in \mathcal{D}(S)\\& = \half \norm{L^2(\partial \Omega)}{S(u)-y_{meas}}^2 + \alpha TV(u) \quad \forall u \in \mathcal{D}(S).
\end{aligned}
\label{eq:comput}
\end{equation}
Hence, $\bar{u}$ is a solution of problem \eqref{eq:minreg}. 
In order to prove that also $TV(u_{k_n}) \rightarrow TV(\bar{u})$, first consider that, according to \eqref{eq:comput}, 
\[
\begin{aligned}
J_{reg}(\bar{u}) &\leq \liminf_n \left(\half \norm{L^2(\partial \Omega)}{S(u_{k_n})-y_{k_n}}^2 + \alpha TV(u_{k_n}) \right) \\&\leq \lim_n \left(\half \norm{L^2(\partial \Omega)}{S(u_{k_n})-y_{k_n}}^2 + \alpha TV(u_{k_n}) \right) = J_{reg}(\bar{u}),
\end{aligned}
\]
hence
\[
\lim_n \left(\half \norm{L^2(\partial \Omega)}{S(u_{k_n})-y_{n_k}}^2 + \alpha TV(u_{k_n}) \right) = \half \norm{L^2(\partial \Omega)}{S(\bar{u})-y_{meas}}^2 + \alpha TV(\bar{u}).
\]
In addition, thanks to the continuity of $S$, the first term in the sum admits a limit, i.e.:
\[
\lim_n \half \norm{L^2(\partial \Omega)}{S(u_{k_n})-y_{k_n}}^2  = \half \norm{L^2(\partial \Omega)}{S(\bar{u})-y_{meas}}^2 ,
\]
which eventually implies that also $TV(u_{k_n}) \rightarrow TV(\bar{u})$.   
\end{proof}

We finally state and prove the following result regarding asymptotic behaviour of the minimum of $J_{reg}$ when $\alpha \rightarrow 0$.
\begin{proposition}
Consider a sequence $\{ \alpha_k \}$ s.t. $\alpha_k \rightarrow 0$, and define the sequence $\{u_k\}$ of the solutions of \eqref{eq:minreg} with the same datum $y_{meas}$ but different weights $\alpha_k$. Suppose there exists (at least) one solution of the inverse problem \eqref{eq:inv}. Then, $\{u_k\}$ admits a convergent subsequence with respect to the $L^1(\Omega)$ norm and the limit $u$ is a minimum-variation solution of the inverse problem, i.e. $S(u)|_{\partial \Omega} = y_{meas}$ and $TV(u) \leq TV(\widetilde{u})$ $\forall \widetilde{u}$ s.t. $S(\widetilde{u})|_{\partial \Omega} = y_{meas}$.
\label{prop:convergence}
\end{proposition}
\begin{proof}
Let $u^{\dag}$ be a solution of the inverse problem. By definition of $u_k$,
$$
\half \norm{L^2(\partial \Omega)}{S(u_k) - y_{meas}}^2 + \alpha_k TV(u_k) \leq \half \norm{L^2(\partial \Omega)}{S(u^\dag) - y_{meas}}^2 + \alpha_k TV(u^\dag) = \alpha_k TV(u^{\dag})
$$
Hence, $\{TV(u_k)\}$ is bounded, and since $\norm{L^1(\Omega)}{u_k} \leq |\Omega|$, ${u_k}$ is also bounded in $BV(\Omega)$ and there exists a subsequence (still denoted as $u_k$) and $u \in X_{0,1}$ s.t. $u_k \xrightharpoonup{BV} u$. Moreover, $\norm{L^2(\partial \Omega)}{S(u_k)|_{\partial \Omega} -y_{meas}}\rightarrow 0$, which implies that $u$ is a solution of the inverse problem \eqref{eq:inv}, and 
$$
TV(u_k) \leq TV(u^{\dag}) \quad \Rightarrow \quad \limsup_k TV(u_k) \leq TV(u^{\dag})
$$
The lower semicontinuity of the $BV$ norm with respect to the weak convergence, together with the continuity of the $L^1$ norm, implies that
$$
 TV(u) \leq \liminf_k TV(u_k) \leq \limsup_k TV(u_k) \leq TV(u^{\dag})
$$
for each solution $u^\dag$ of the inverse problem, which eventually implies that $u$ is a minimum-variation solution.
\end{proof}

Notice that, if the minimum-variation solution of problem \eqref{eq:opt} is unique, then the sequence $\{u_k\}$ converges to it. 
\par
 The latter result can be improved by considering small perturbation of the data. By similar arguments as in proof of Proposition \ref{prop:convergence}, one can prove the following
\begin{proposition}
Let $y^{\delta} \in L^2(\partial \Omega)$ s.t. $\norm{L^2(\partial \Omega)}{y^{\delta}-y_{meas}} \leq \delta$ and let $\alpha(\delta)$ be such that $\alpha(\delta) \rightarrow 0$ and $\frac{\delta^2}{\alpha(\delta)} \rightarrow 0$ as $\delta \rightarrow 0$. Suppose there exists at least one solution of the inverse problem \eqref{eq:inv}. Then, every sequence $\{u_{\alpha_k}^{\delta_k}\}$, with $\delta_k \rightarrow 0$, $\alpha_k = \alpha(\delta_k)$ and $u_{\alpha_k}^{\delta_k}$ solution of \eqref{eq:minreg} corresponding to $\alpha_k$ and $y^{\delta_k}$, has a converging subsequence with respect to the $L^1(\Omega)$ norm. The limit $u$ of every convergent subsequence is a minimum-variation solution of the inverse problem.
\label{prop:convergence2}
\end{proposition}
\begin{proof} 
Consider a solution $u^{\dag}$ of the inverse problem. By definition of $u_{\alpha_k}^{\delta_k}$,
\begin{equation}
\half \norm{L^2(\partial \Omega)}{S(u_{\alpha_k}^{\delta_k}) - y^{\delta_k}}^2 + \alpha_k TV(u_{\alpha_k}^{\delta_k}) \leq \half \norm{L^2(\partial \Omega)}{S(u^\dag) - y^{\delta_k}}^2 + \alpha_k TV(u^\dag) \leq \delta_k^2 + \alpha_k TV(u^{\dag})
\label{eq:1.6.1}
\end{equation}
In particular, 
\begin{equation}
TV(u_{\alpha_k}^{\delta_k}) \leq \frac{\delta_k^2}{\alpha_k} + TV(u^{\dag}),
\label{eq:1.6.2}
\end{equation}
hence $\{u_{\alpha_k}^{\delta_k}\}$ is bounded in $BV(\Omega)$ and admits a subsequence (denoted by the same index $k$) such that $\exists u \in X_{0,1}$: $u_{\alpha_k}^{\delta_k}\xrightharpoonup{BV} u$. Passing to the limit in \eqref{eq:1.6.1} as $k \rightarrow +\infty$,
$$
\norm{L^2(\partial \Omega)}{S(u_{\alpha_k}^{\delta_k}) - y^{\delta_k}}^2 \rightarrow 0,
$$
hence also
$$
\norm{L^2(\partial \Omega)}{S(u_{\alpha_k}^{\delta_k}) - y_{meas}}^2 \leq \norm{L^2(\partial \Omega)}{S(u_{\alpha_k}^{\delta_k}) - y^{\delta_k}}^2 + \norm{L^2(\partial \Omega)}{y^{\delta_k} - y_{meas}}^2 \rightarrow 0
$$
and by continuity of the solution map, we have that $S(u)|_{\partial \Omega} = y_{meas}$, which implies that $u$ is a solution of the inverse problem. By lower semi continuity of the BV norm (hence of the total variation) with respect the weak convergence and from inequality \eqref{eq:1.6.2},
$$
 TV(u) \leq \liminf_k TV(u_{\alpha_k}^{\delta_k}) \leq \limsup_k TV(u_{\alpha_k}^{\delta_k}) \leq TV(u^{\dag}),
$$
which allows to conclude that $u$ is also a minimum-variation solution of the inverse problem.
\end{proof}
Thanks to the results outlined in this section, it is possible to assess the stability of the regularized inverse problem:
\[
\argmin_{u \in X_{0,1}} J_{reg}(u); \quad J_{reg}(u) = \half \norm{L^2(\partial \Omega)}{S(u) - y_{meas}}^2 + \alpha TV(u),
\]
In principle, one can obtain successive approximations of the solution of the inverse problem \eqref{eq:inv} by solving the minimization problem \eqref{eq:minreg} with fixed $\alpha > 0$. However, this approach would require to deal with several technical difficulties, namely the non-differentiability of the cost functional $J_{reg}$ and the non-convexity of the space $X_{0,1}$. This will be object of future research by adapting, e.g., the technique in \cite{art:bartels2012} to the present context. However, in the sequel we follow a different strategy, namely introducing a phase-field relaxation of problem \eqref{eq:minreg}. 

\section{Relaxation}
\label{sec:relax}
In this section, we formulate the phase-field relaxation of the optimization problem \eqref{eq:minreg}. Fixed a relaxation parameter $\varepsilon > 0$, the Total Variation term in the expression of $J_{reg}$ is replaced with a smooth approximation, known as Ginzburg-Landau energy or Modica-Mortola functional; moreover, the minimization is set in a space of more regular functions.
We follow a similar strategy as in \cite{art:deckelnick}, with the additional difficulty of the non-linearity of the direct problem. In particular, we prove the main properties of the relaxed problem: existence of a solution is assessed in Proposition \ref{prop:exist2} and convergence to the (sharp) initial problem \eqref{eq:minreg} as $\varepsilon \rightarrow 0$ is proved in Proposition \ref{prop:conv2}. Moreover, in Proposition \ref{prop:optcond}, we describe the optimality conditions associated to the minimization problem, and compute the Frechét derivative of the relaxed cost functional, which is useful for the reconstruction purposes.
\par
Consider $u \in \mathcal{K} = H^1(\Omega;[0,1]) = \{ v \in H^1(\Omega): \ 0 \leq v \leq 1 \ a.e.\}$ and, for every $\varepsilon >0$, introduce the optimization problem:
\begin{equation}
\argmin_{u \in \mathcal{K}} J_{\varepsilon}(u); \quad J_{\varepsilon}(u) = \half \norm{L^2(\partial \Omega)}{S(u)-y_{meas}}^2 + \alpha \int_{\Omega}\left( \varepsilon|\nabla u|^2 + \frac{1}{\varepsilon}u(1-u)\right).
\label{eq:minrel}
\end{equation}

The first theoretical result that is possible to prove guarantees the existence of a solution of the relaxed problem, employing classical techniques of Calculus of Variations.


\begin{proposition}
	For every fixed $\varepsilon > 0$, the minimization problem \eqref{eq:minrel} has a solution $\ue \in \mathcal{K}$.
	\label{prop:exist2}
\end{proposition}
\begin{proof}
	Fix $\varepsilon > 0$ and consider a minimizing sequence $\{ u_k \} \subset \mathcal{K}$ (we omit the dependence of $u_k$ on $\varepsilon$). By definition of the minimizing sequence, $J_\varepsilon(u_k) \leq M$ indepentently of $k$, which directly implies that also $\norm{L^2(\Omega)}{\nabla u_k}^2$ is bounded. Moreover, being $u_k \in \mathcal{K}$, $0\leq u_k \leq 1$ a.e., thus $\norm{L^2(\Omega)}{u_k}^2 \leq |\Omega|$ and it is possible to conclude that $\{u_k\}$ is bounded in $H^1(\Omega)$. Thanks to weak compactness of $H^1$, there exist $\ue \in H^1(\Omega)$ and a subsequence $\{ u_{k_n} \}$ s.t. $u_{k_n} \xrightharpoonup{H^1} \ue$, hence $u_{k_n} \xrightarrow{L^2} \ue$. The strong $L^2$ convergence implies (up to a subsequence) pointwise convergence a.e., which allows to conclude (together with the dominated convergence theorem, since $u_{k_n}(1-u_{k_n}) \leq 1/2$) that
	$$
	\int_{\Omega} u_{k_n}(1-u_{k_n}) \rightarrow \int_{\Omega} \ue(1-\ue).
	$$
	Moreover, by the lower semicontinuity of the $H^1$ norm with respect to the weak convergence, and by the compact embedding in $L^2$,
	$$
	\begin{aligned}
	\norm{H^1(\Omega)}{\ue}^2 &\leq \liminf_n  \norm{H^1(\Omega)}{u_{k_n}}^2 \\
	\norm{L^2(\Omega)}{\ue}^2 + \norm{L^2(\Omega)}{\nabla \ue}^2 &\leq  \lim_n \norm{L^2(\Omega)}{u_{k_n}}^2 + \liminf_n \norm{L^2(\Omega)}{\nabla u_{k_n}}^2 \\
	\norm{L^2(\Omega)}{\nabla \ue}^2 &\leq \liminf_n \norm{L^2(\Omega)}{\nabla u_{k_n}}^2.
	\end{aligned}
	$$
	Moreover, using the continuity of the solution map $S$ with respect to the $L^1$ convergence, we can conclude that
	$$
	J_\varepsilon(\ue) \leq \liminf_n J_\varepsilon(u_{k_n}).
	$$
	Finally, by pointwise convergence, $0 \leq \ue \leq 1$ a.e., hence $\ue$ is a minimum of $\Je$ in $\mathcal{K}$.
\end{proof}
The asymptotic behaviour of the phase-field problem when $\varepsilon \rightarrow 0$ is investigated in the next two propositions. First, we prove that the relaxed cost functional $J_\varepsilon$ converges to $J_{reg}$ in the sense of the $\Gamma$-convergence, which will naturally entail the convergence of the corresponding minimizers. Before stating the next result, we have to introduce the space $X$ of the Lebesgue-measurable functions over $\Omega$ endowed with the $L^1(\Omega)$ norm and consider the following extension of the cost funtionals: problem \eqref{eq:minreg} is replaced by
\begin{equation}
\argmin_{u \in X} \tilde{J}(u); \quad \tilde{J}(u) = 
\left\{
	\begin{aligned}
	J_{reg}(u) & \quad \textit{if $u \in X_{0,1}=BV(\Omega;\{0,1\})$}\\
	\infty & \quad \textit{otherwise}
	\end{aligned}
\right.
\label{eq:minregbis}
\end{equation}
whereas \eqref{eq:minrel} is replaced by
\begin{equation}
\argmin_{u \in X} \tilde{J}_{\varepsilon}(u); \quad \tilde{J}_{\varepsilon}(u) = 
\left\{
	\begin{aligned}
	J_\varepsilon(u) & \quad \textit{if $u \in \mathcal{K}$}\\
	\infty & \quad \textit{otherwise}
	\end{aligned}
\right.
\label{eq:minrel2}
\end{equation}
It is now possible to formulate the convergence result, whose proof can be easily obtained by adapting the one of \cite[Theorem 6.1]{art:deckelnick}.
\begin{proposition}
	Consider a sequence $\{\varepsilon_k\}$ s.t. $\varepsilon_k \rightarrow 0$. Then, the functionals $\tilde{J}_{\varepsilon_k}$ converge to $\tilde{J}$ in $X$ in the sense of the $\Gamma-$convergence.
\end{proposition}
Finally, from the compactness result in \cite[Proposition 4.1]{baldo1990} and applying the definition of $\Gamma$-convergence, it is easy to prove the following convergence result for the solutions of \eqref{eq:minrel}.
\begin{proposition}
Consider a sequence $\{\varepsilon_k\}$ s.t. $\varepsilon_k \rightarrow 0$ and let $\{u_{\varepsilon_k}\}$ be the sequence of the respective minimizers of the functionals $\{\tilde{J}_{\varepsilon_k}\}$. Then, there exists a subsequence, still denoted as $\{\varepsilon_k\}$ and a function $u \in X_{0,1}$ such that $u_{\varepsilon_k} \rightarrow u$ in $L^1$ and $u$ is a solution of \eqref{eq:minreg}.
	\label{prop:conv2}
\end{proposition}
%

\subsection{Optimality conditions}
\label{sec:oc}
We can now provide an expression for the optimality condition associated with the minimization problem \eqref{eq:minrel}, which is formulated as a variational inequality involving the Fréchet derivative of $\Je$.
\begin{proposition}
	Consider the solution map $S:\mathcal{K} \rightarrow H^1(\Omega)$ and let $f \in L^2(\Omega)$ satisfy the hypotheses in Proposition \ref{prop:geq} or \ref{prop:geq2}: for every $\varepsilon>0$, the operators $S$ and $\Je$ are Fréchet-differentiable on $\mathcal{K} \subset L^{\infty}(\Omega) \cap H^1(\Omega)$ and a minimizer $\ue$ of $\Je$ satisfies the variational inequality: 
	\begin{equation}
	J'_\varepsilon(u_\varepsilon)[v - u_\varepsilon] \geq 0 \qquad \forall v \in \mathcal{K},
	\label{eq:OC}
	\end{equation}
	being
	\begin{equation}
		J'_\varepsilon(u)[\vartheta] = \int_{\Omega}(1-k)\vartheta \nabla S(u) \cdot \nabla p + \int_{\Omega} \vartheta S(u)^3 p +  2 \alpha \varepsilon \int_{\Omega}\nabla u \cdot \nabla \vartheta  + \frac{\alpha}{\varepsilon}	\int_{\Omega}(1-2u)\vartheta;
	\label{eq:VI}
	\end{equation}
where $\vartheta \in \mathcal{K}-\mathcal{K}=\{v \ s.t. \ v=u_1-u_2, \ u_1,u_2 \in \mathcal{K}\}$ and $p$ is the solution of the \textit{adjoint problem}:  
\begin{equation}
	\int_{\Omega} a(u)\nabla p \cdot \nabla \psi + \int_{\Omega} 3b(u)S(u)^2 p \psi = \int_{\partial \Omega}{(S(u)-y_{meas})\psi} \qquad \forall \psi \in H^1(\Omega).
\label{eq:adjoint}
\end{equation}
\label{prop:optcond}
\end{proposition}

\begin{proof}
First of all we need to prove that $S$ is Fréchet differentiable in $L^{\infty}(\Omega)$: in particular, we claim that for $\vartheta \in L^{\infty}(\Omega)\cap (\mathcal{K}-\mathcal{K})$ it holds that $S'(u)[\vartheta] = S_*$, being $S_*$ the solution in $H^1(\Omega)$ of
\begin{equation}
	\int_{\Omega} a(u)\nabla S_* \nabla \varphi + \int_{\Omega} b(u) 3S(u)^2 S_* \varphi = \int_{\Omega} (1-k)\vartheta \nabla S \nabla \varphi + \int_{\Omega} \vartheta S(u)^3 \varphi \quad \forall \varphi \in H^1(\Omega),
	\label{eq:derprob}
\end{equation} 
namely, that
\begin{equation}
	\norm{H^1(\Omega)}{S(u+\vartheta) - S(u) - S_*} = o(\norm{L^{\infty}(\Omega)}{\vartheta}).
\label{eq:Frechet}
\end{equation}
First we show that if $\vartheta \in L^{\infty}(\Omega)\cap (\mathcal{K}-\mathcal{K})$, then $\norm{H^1(\Omega)}{S(u+\vartheta)-S(u)} \leq C \norm{L^{\infty}(\Omega)}{\vartheta}$. 
Indeed, the difference $w = S(u+\vartheta)-S(u)$ satisfies
\begin{equation}
\begin{aligned}
	\int_{\Omega} a(u+\vartheta)\nabla w \nabla \varphi + \int_{\Omega} b(u+\vartheta) q w \varphi =& -\int_{\Omega} (a(u+\vartheta)-a(u))\nabla S(u) \nabla \varphi \\
	&- \int_{\Omega} (b(u+\vartheta)-b(u)) S(u)^3 \varphi \qquad \forall \varphi \in H^1(\Omega),
	\end{aligned}
\label{eq:difference}
\end{equation}
with $q = S(u+\vartheta)^2 + S(u)S(u+\vartheta) +S(u)^2$. Since $a(u+\vartheta)-a(u)= -(1-k)\vartheta$ and $b(u+\vartheta)-b(u) = -\vartheta$, and arguing as in the proof of Proposition \ref{prop:continuity}, take $\varphi = w$ in \eqref{eq:difference}: we obtain
\[
k\norm{L^2}{\nabla w}^2 + \frac{3}{4}\int_\Omega b(u+\vartheta)S(u)^2w^2 \leq (1-k)\norm{L^\infty}{\vartheta} \norm{L^2}{\nabla S(u)} \norm{L^2}{\nabla w} + \norm{L^2}{S(u)^3} \norm{L^2}{w}\norm{L^{\infty}}{\vartheta}
\]
and again by Proposition \ref{prop:geq}
\[
\begin{aligned}
k \norm{L^2}{\nabla w}^2 + \frac{3}{4}Q \norm{L^2(\Omega^*)}{w}^2 \leq &(1-k)\norm{L^\infty}{\vartheta} \norm{L^2}{\nabla S(u)} \norm{L^2}{\nabla w} + \norm{L^\infty}{\vartheta} \norm{L^2}{S(u)^3}\norm{L^2}{w} \\&+ \frac{3}{4} \norm{L^\infty}{\vartheta} \norm{L^2}{S(u)^2}\norm{L^2}{w}.
\end{aligned}
\]
By \eqref{eq:Poincare} and by Sobolev inequality, eventually
\[
	\norm{H^1(\Omega)}{w}^2 \leq C \norm{H^1(\Omega)}{S(u)}\norm{H^1(\Omega)}{w}\norm{L^\infty}{\vartheta},
\]
hence $\norm{H^1(\Omega)}{S(u+\vartheta)-S(u)} = O(\norm{L^\infty(\Omega)}{\vartheta})$.
\par
Take now \eqref{eq:difference} and subtract \eqref{eq:derprob}. Define $r = S(u+\vartheta)-S(u) - S_*$: it holds that
\[
\begin{aligned}
	\int_{\Omega} a(u)\nabla r \nabla \varphi + \int_{\Omega} b(u) 3 S(u)^2 r \varphi =& \int_{\Omega} (a(u + \vartheta) - a(u))\nabla w \cdot \nabla \varphi \\
	&+ \int_{\Omega} (b(u+\vartheta)q - 3b(u)S(u)^2 )w \varphi \qquad \forall \varphi \in H^1(\Omega). 
\end{aligned}
\]	
The second integral in the latter sum can be split as follows:
\[
\int_{\Omega} (b(u+\vartheta)q - 3b(u)S(u)^2 )w \varphi = \int_{\Omega} (b(u+\vartheta)-b(u))q w \varphi + \int_{\Omega} (q - 3S(u)^2 )b(u)w \varphi,
\]
and in particular $q - 3S(u)^2 = S(u+\vartheta)^2 + S(u+\vartheta)S(u) - 2S(u)^2 = hw$, where $h = S(u+\vartheta)+2S(u) \in H^1(\Omega)$. Hence, chosen $\varphi = r$ and exploiting again the Poincaré inequality in Lemma \ref{Lemma1} and the H{\"o}lder inequality:
\[
\begin{aligned}
	\frac{1}{C}\norm{H^1}{r}^2 \leq &  k\norm{L^2}{\nabla r}^2 + Q \norm{L^2(\Omega^*)}{r} \leq (1-k)\norm{L^{\infty}}{\vartheta} \norm{L^2}{\nabla w}\norm{L^2}{\nabla r} \\
	& + \norm{L^{\infty}}{\vartheta} \norm{L^4}{q}\norm{L^2}{w}\norm{L^4}{r} +  \norm{L^4}{h}\norm{L^4}{w}^2 \norm{L^4}{r}\\
	& \leq \left(  (1-k)\norm{L^{\infty}}{\vartheta}\norm{H^1}{w} + \norm{H^1}{q}\norm{L^{\infty}}{\vartheta}\norm{H^1}{w} +  \norm{H^1}{h}\norm{H^1}{w}^2 \right)\norm{H^1}{r}.
\end{aligned}
\]
It follows eventually that $\norm{H^1(\Omega)}{r} \leq C \norm{L^{\infty}}{\vartheta}^2 = o(\norm{L^{\infty}}{\vartheta})$, which guarantees that $S_*=S'(u)[\vartheta]$.

The last step is to provide an expression of the Fréchet derivative of $\Je$. Exploiting the fact that $S$ is differentiable, we can compute the expression of $\Je'(u)$ through the \textit{chain rule}:
\begin{equation}
	\Je'(u)[\vartheta] = \int_{\partial \Omega} (S(u)-y_0) S'(u)[\vartheta] + \alpha \int_{\Omega}\left( 2\varepsilon \nabla u \nabla \vartheta + \frac{1}{\varepsilon}(1-2u)\vartheta \right).
\label{eq:first}
\end{equation}
Finally, thanks to the expression of the adjoint problem,
\[
\begin{aligned}
	\int_{\partial \Omega} (S(u)-y_0) S'(u)[\vartheta] =& \int_{\partial \Omega} (S(u)-y_0)S_*  = \int_{\Omega} a(u)\nabla p \cdot \nabla S_* + \int_{\Omega} 3S(u)^2 p S_* =  \\ 
		\textit{(by definition of $S_*$)} =& \int_{\Omega}(1-k)\vartheta \nabla S(u) \cdot\nabla p + \int_{\Omega} \vartheta S(u)^3p ,
\end{aligned}
\]
and hence:
\[
 \Je'(u)[\vartheta] = \int_{\Omega}(1-k)\vartheta \nabla S(u)\cdot \nabla p + \int_{\Omega} \vartheta S(u)^3p + \alpha \int_{\Omega}\left( 2\varepsilon \nabla u \cdot \nabla \vartheta + \frac{1}{\varepsilon}(1-2u)\vartheta \right).
\] 
It is eventually a standard argument that, being $\Je$ a continuous and Frechét differentiable functional on a convex subset $\mathcal{K}$ of the Banach space $H^1(\Omega)$, the optimality conditions for the optimization problem \eqref{eq:minrel} are expressed by the variational inequality \eqref{eq:OC}.
\end{proof}

\section{Discretization and reconstruction algorithm}
\label{sec:discr}
For a fixed $\varepsilon>0$, we now introduce a discrete formulation of problem \eqref{eq:minrel} in order to define a numerical reconstruction algorithm and compute an approximated solution of the inverse problem. 

\par In what follows, we consider $\Omega$ to be polygonal, in order to avoid a discretization error involving the geometry of the domain. Let $\Tau_h$ be a shape regular triangulation of $\Omega$ and define $V_h \subset H^1(\Omega)$:
\[
	V_h = \{ v_h \in C(\bar{\Omega}), v_h|_K \in \mathbb{P}_1(K) \text{ } \forall K \in \Tau_h \}; \qquad \mathcal{K}_h = V_h \cap \mathcal{K}.
\]
It is well known, from the Clément interpolation theory (see e.g. \cite{book:brenner2007}), that for every $w \in H^1(\Omega)$ there exists a sequence $\{w_h\}$ such that 
\begin{equation}
w_h \in V_h \qquad w_h \xrightarrow{H^1} w \qquad \text{as $h \rightarrow 0$}. 
\label{eq:Clement}
\end{equation}
For every fixed $h>0$, we define the solution map $S_h: \mathcal{K} \rightarrow V_h$, where $S_h(u)$ solves
\[
\int_{\Omega} a(u) \nabla S_h(u) \nabla v_h + \int_{\Omega} b(u) S_h(u)^3 v_h = \int_{\Omega} f_h v_h \quad v_h \in V_h,
\]
being $f_h$ the Clément interpolator of $f$ in the space $V_h$, hence $\norm{H^1(\Omega)}{f_h - f} \rightarrow 0$.

\subsection{Convergence analysis as $h \rightarrow 0$}
The present section is devoted to the numerical analysis of the discretized problem: the convergence of the approximated solution of the direct problem is studied, taking into account the difficulties implied by the non linear term. Moreover, the existence and convergence of minimizers of the discrete cost functional is analysed.
The following result, which is preliminary for the proof of the convergence of the approximated solutions to the exact one, can be proved by resorting at the techniques of \cite[Theorem 2.1]{art:ciarlet}. For completeness we briefly report the proof.
\begin{lemma}
Let $f \in L^2(\Omega)$ satisfy the hypotheses in Proposition \ref{prop:geq} or \ref{prop:geq2}; then, for every $u \in \mathcal{K}$, $S_h(u) \rightarrow S(u)$ strongly in $H^1(\Omega)$.
\label{prop:ciarlet}
\end{lemma}
\begin{proof}
As in the proof of Proposition \ref{prop:wellpos}, for a fixed $u \in \mathcal{K}$ we define the operator $T:H^1(\Omega) \rightarrow (H^1(\Omega))^*$ such that 
\[
	\langle T(y), \varphi \rangle = \int_{\Omega} a(u)\nabla y \nabla \varphi + \int_{\Omega} b(u) y^3 \varphi;
\]
then $y_h = S_h(u)$ and $y = S(u)$ are respectively the solutions of the equations
\[
\begin{aligned}
 \langle T(y_h), \varphi \rangle &= \int_\Omega f_h \varphi \quad & \forall \varphi \in V_h; \qquad \qquad
 \langle T(y), \varphi \rangle &= \int_\Omega f \varphi \quad & \forall \varphi \in H^1(\Omega).
\end{aligned}
\]
It is easy to prove that 
\begin{equation}
	\langle T(y_h) - T(y), y_h - y \rangle \geq C(k,y) \norm{H^1}{y_h-y}^2;
\label{eq:ellipticity}
\end{equation}
indeed, thanks to Lemma \ref{Lemma1} and Proposition \ref{prop:geq},
\[
	\langle T(y_h) - T(y), y_h - y \rangle = \int_\Omega a(u) |\nabla (y_h - y)|^2 + \int_\Omega b(u)(y_h - y)^2 (y_h^2 + y_h y + y^2) \geq C \norm{H^1}{y_h - y}^2.
\]
Thanks to \eqref{eq:ellipticity}, consider a generic $w_h \in V_h$,
\[
\begin{aligned}
\norm{H^1}{y_h - y}^2 &\leq \langle T(y_h) - T(y), y_h - y \rangle = \langle T(y_h) - T(y), w_h - y \rangle + \langle T(y_h) - T(y), y_h - w_h \rangle \\
& \leq \langle T(y_h) - T(y), w_h - y \rangle + \int_\Omega (f_h - f)(y_h - w_h) \\
& \leq K \norm{H^1}{w_h - y} \norm{H^1}{y_h - y} + \norm{H^1}{f_h - f}\norm{H^1}{y_h - w_h},
\end{aligned}
\]
where $K$ is the (local) Lipschitz constant of $T$ (see Proposition \ref{prop:wellpos}). Hence:
\[
\begin{aligned}
\norm{H^1}{y_h - y} &\leq K \norm{H^1}{w_h - y} + \half \sqrt{K^2 \norm{H^1}{w_h - y}^2 + \norm{H^1}{y_h - w_h}^2 \norm{H^1}{f_h - f}^2} \\
	&\leq C_1 \norm{H^1}{w_h - y} + C_2 \norm{H^1}{y_h - w_h} \norm{H^1}{f_h - f}\\
	&\leq (C_1 + C_2 \norm{H^1}{f_h - f})\norm{H^1}{w_h - y} + C_2 \norm{H^1}{y_h - y} \norm{H^1}{f_h - f}.
\end{aligned}
\]
Since $f_h \xrightarrow{H^1} f$, we can choose $h$ sufficiently small s.t. $C_2 \norm{H^1}{f_h - f} \leq \half$, hence:
\[
\norm{H^1}{y_h - y} \leq 2 \left(C_1 + \half \right) \norm{H^1}{w_h - y},
\]
and since the latter inequality holds for each $w_h \in H^1(\Omega)$, it holds:
\[
\norm{H^1(\Omega)}{y_h - y} \leq C \inf_{w_h \in V_h}\norm{H^1(\Omega)}{w_h - y}. 
\]
Finally, exploiting \eqref{eq:Clement}, we conclude the thesis.
\end{proof}

The convergence of the solution of the discrete direct problem to the continuous one is an immediate consequence of Lemma \ref{prop:ciarlet} and of the continuity of the map $S_h$ in the space $V_h$, which can be assessed analogously to the proof of Proposition \ref{prop:continuity}.
\begin{proposition}
	Let $\{h_k\}, \{u_k\}$ be two sequences such that $h_k \rightarrow 0$, $u_k \in \mathcal{K}_{h_k}$ and $u_k \xrightarrow{L^1} u$ and $u$ is not identically equal to 1. Then $S_{h_k}(u_k) \xrightarrow{H^1} S(u)$.
\label{cor:1}
\end{proposition}
Define the discrete cost functional, $J_{\varepsilon,h}: \mathcal{K}_h \rightarrow \R$
\begin{equation}
 J_{\varepsilon,h}(u_h) = \half \norm{L^2(\partial \Omega)}{S_h(u_h)-y_{meas,h}}^2 + \alpha \int_{\Omega}\left( \varepsilon|\nabla u_h|^2 + \frac{1}{\varepsilon}u_h(1-u_h)\right),
\label{eq:Jeh}
\end{equation}
being $y_{meas,h}$ the best approximation of the boundary datum $y_{meas}$ in the space of the traces of $V_h$ functions.
The existence of minimizers of the discrete functionals $J_{\varepsilon,h}$ is stated in the following proposition, together with an asymptotic analysis as $h \rightarrow 0$. Taking advantage of Proposition \ref{cor:1}, the proof is analogous to the one of \cite[Theorem 3.2]{art:deckelnick}.
\begin{proposition}
For each $h>0$, there exists $u_h \in \mathcal{K}_h$ such that $J_{\varepsilon,h}(u_h) = min_{v_h \in \mathcal{K}_h}J_{\varepsilon,h}(v_h)$. Every sequence $\{u_{h_k}\}$ s.t. $\lim_{k\rightarrow \infty}h_k = 0$ admits a subsequence that converges in $H^1(\Omega)$ to a minimum of the cost functional $J_\varepsilon$.
\label{prop:minJeh}
\end{proposition}

The strategy we adopt in order to minimize the discrete cost functional $J_{\varepsilon,h}$ is to search for a function $u_h$ satisfying discrete optimality conditions, which can be obtained as in section \ref{sec:oc}:
\begin{equation}
	J'_{\varepsilon,h}(u_h)[v_h - u_h] \geq 0 \quad \forall v_h \in \mathcal{K}_h
\label{eq:discrVI}
\end{equation}
where for each $\theta_h \in \mathcal{K}_h - \mathcal{K}_h := \{\theta_h = w_h - v_h; \ w_h, v_h \in \mathcal{K}_h\}$ it holds
\begin{equation}
	J'_{\varepsilon,h}(u_h)[\vartheta_h] = \int_{\Omega}(1-k)\vartheta_h \nabla S_h(u_h) \cdot \nabla p_h + \int_{\Omega}\vartheta_h S_h(u_h)^3 p_h +  2 \alpha \varepsilon \int_{\Omega}\nabla u_h \cdot \nabla \vartheta_h  + \frac{\alpha}{\varepsilon}	\int_{\Omega}(1-2u_h)\vartheta_h,
\label{eq:discrJ1}
\end{equation}
where $p_h$ is the solution in $V_h$ of the adjoint problem \eqref{eq:adjoint} associated to $u_h$. 

It is finally possible to demonstrate the convergence of critical points of the discrete functionals $J_{\varepsilon,h}$ (i.e., functions in $\mathcal{K}_h$ satisfying \eqref{eq:discrVI}) to a critical point of the continuous on, $J_\varepsilon$. The proof can be adapted from the one of \cite[Theorem 3.2]{art:deckelnick}.
\begin{proposition}
Consider a sequence $\{h_k\}$ s.t. $h_{k} \rightarrow 0$ and for every $k$ denote as $u_k$ a solution of the discrete variational inequality \eqref{eq:discrVI}. Then there exists a subsequence of $\{u_k\}$ that converges a.e and in $H^1(\Omega)$ to a solution $u$ of the continuous variational inequality \eqref{eq:VI}
\label{prop:discrVI}
\end{proposition}

\subsection{Reconstruction algorithm: a Parabolic Obstacle Problem approach}
\label{POP}
The necessary optimality conditions that have been stated in Proposition \ref{prop:optcond}, together with the expression of the Fréchet derivative of the cost functional reported in \eqref{eq:VI} allow to define a Parabolic Obstacle problem, which consists in a very common strategy in order to search for a solution of optimization problems in a phase-field approach. In this section we give a continuous formulation of the problem, and provide a formal proof of its desired properties. We then introduce a numerical discretization of the problem and rigorously prove the main convergence results. 
\par
The core of the proposed approach is to rely on a parabolic problem whose solution $u(\cdot,t)$ converges, as the fictitious time variable tends to $+\infty$, to an asymptotic state $u_{\infty}$ satisfying the continuous optimality conidtions \eqref{eq:VI}. The problem can be formulated as follows, for a fixed $\varepsilon > 0$: let $u$ be the solution of 
\begin{equation}
	\left\{
	\begin{aligned}
		\int_{\Omega} \partial_t u (v-u) + \Je'(u)[v-u] &\geq 0 \qquad \forall v \in \mathcal{K}, \quad t \in(0,+\infty)\\
		u(\cdot,0) &= u_0 \in \mathcal{K}
	\end{aligned}
	\right.
\label{eq:POP}
\end{equation}
The theoretical analysis of the latter problem is beyond the purposes of this work, and would require to deal with the severe non-linearity of the expression of $\Je'(u)$. We reduce ourselves to formally report the expected properties of the Parabolic Obstacle Problem and then analyse in detail its discretised version.
The motivation for the introduction of the Parabolic Obstacle Problem is twofold:
\begin{itemize}
	\item \textit{the evaluation of the cost functional along the solution of problem \eqref{eq:POP} is a decreasing function of time.} Indeed, 
	\[
\begin{aligned}
\frac{d}{dt} \Je(u(\cdot,t)) 
\leq -\norm{L^2(\Omega)}{\partial_t u(\cdot,t)}^2 \leq 0.
\end{aligned}
\]
\item \textit{ As $t \rightarrow +\infty$, the solution $u(\cdot,t)$ converges to $u_\infty \in H^1(\Omega)$, which satisfies the optimality conditions \eqref{eq:OC}}. 
\end{itemize}

We now provide a complete discretization of the Parabolic Obstacle Problem 
by setting \eqref{eq:POP} in the discrete spaces $\mathcal{K}_h$ and $V_h$, 
and by considering a semi-implicit one-step scheme for the time updating, as in \cite{art:deckelnick}: i.e., by treating explicitly the nonlinear terms and implicitly the linear ones. We obtain that the approximate solution $\{u_h^n\}_{n\in \N} \subset V_h,$ $u_h^n\approx u(\cdot,t^n)$ is computed as:
\begin{equation}
	\left\{
	\begin{aligned}
	u_h^0 &= u_0 \in \mathcal{K}_h \qquad \textit{(a prescribed initial datum)}\\
	u_h^{n+1}& \in \mathcal{K}_h: \int_{\Omega}(u_h^{n+1}-u_h^n)(v_h-u_h^{n+1}) + \tau_n \int_{\Omega}(1-k)\nabla S_h(u_h^n)\cdot \nabla p_h^n (v_h - u_h^{n+1}) \\
	&\qquad \qquad+\tau_n \int_{\Omega}S_h(u_h^n)^3 p_h^n (v_h - u_h^{n+1}) + 2 \tau_n \alpha \varepsilon \int_{\Omega}\nabla u_h^{n+1} \cdot \nabla(v_h - u_h^{n+1}) \\
	&\qquad \qquad+\tau_n \alpha \frac{1}{\varepsilon} \int_{\Omega}(1-2u_h^n)(v_h - u_h^{n+1})\geq 0 \quad \forall v_h \in \mathcal{K}_h, \ n=0,1,\ldots
	\end{aligned}
\right.
\label{eq:fulldiscrPOP}
\end{equation}

For the fully-discretized problem \eqref{eq:fulldiscrPOP}, it is possible to prove rigorously the properties that we have formally stated for the continuous one; in particular, the convergence of the sequence $\{u_h^n\}$ to a critical point of the discrete cost functional $J_{\varepsilon,h}$. The following preliminary result is necessary for the proof of the main property:
\begin{lemma}
For each $n > 0$, there exists a positive constant $\mathcal{B}_n = \mathcal{B}_n(\Omega, h, k, \norm{H^1}{p_h^n}, \norm{H^1}{y_h^n}, \norm{H^1}{y_h^{n+1}})$ such that, provided that $\tau_n \leq \mathcal{B}_n$ it holds that:
	\begin{equation}
	\norm{L^2}{u_h^{n+1} - u_h^n}^2 + J_{\varepsilon,h}(u_h^{n+1}) \leq J_{\varepsilon,h}(u_h^n) \quad n>0.
	\label{eq:Jdecrease}
	\end{equation}
	\label{prop:decrease}
\end{lemma}
\begin{proof}
In the expression of the discrete parabolic obstacle problem \eqref{eq:fulldiscrPOP}, consider $v_h = u_h^n$: via simple computation, we can point out that
\[
\begin{aligned}
\frac{1}{\tau_n}&\norm{L^2}{u_h^{n+1}-u_h^n}^2 + J(u_h^{n+1}) - J(u_h^{n}) + \alpha \varepsilon \norm{L^2}{\nabla (u_h^{n+1} - u_h^n)}^2 + \frac{\alpha}{\varepsilon}\norm{L^2}{u_h^{n+1}-u_h^n}^2 \\
&\leq \int_{\Omega}\left( a(u_h^{n+1}) - a(u_h^{n})\right) \nabla y_h^n \nabla p_h^n + \int_{\Omega}\left( b(u_h^{n+1}) - b(u_h^{n}) \right)(y_h^n)^3 p_h^n \\
&+ \half \norm{L^2(\partial \Omega)}{y_h^{n+1}-y_h^n}^2 + \int_{\partial \Omega} (y_h^{n+1}-y_h^n)(y_h^{n+1} - y_{meas,h}),
\end{aligned}
\]
where $y_h^n = S_h(u_h^n)$ and $y_h^{n+1} = S_h(u_h^{n+1})$.
Moreover, by the expression of the adjoint problem, 
\[
RHS = \half \norm{L^2(\partial \Omega)}{y_h^{n+1}-y_h^n}^2 + \circled{I} + \circled{II},
\]
where
\[
\begin{aligned}
\circled{I} &= \int_{\Omega}\left( a(u_h^{n+1}) - a(u_h^{n})\right) \nabla y_h^n \cdot \nabla p_h^n + \int_\Omega a(u_h^n)\nabla p_h^n \cdot \nabla(y_h^{n+1}-y_h^n) \\
&= \int_{\Omega}\left( a(u_h^n) - a(u_h^{n+1})\right) \nabla (y_h^{n+1} - y_h^n) \cdot \nabla p_h^n  + \int_\Omega a(u_h^{n+1})\nabla y_h^{n+1}\cdot \nabla p_h^n - \int_\Omega a(u_h^{n})\nabla y_h^{n}\cdot \nabla p_h^n;
\end{aligned}
\]
\[
\begin{aligned}
\circled{II} &= \int_{\Omega}\left( b(u_h^{n+1}) - b(u_h^{n})\right) (y_h^n)^3 p_h^n + 3\int_\Omega b(u_h^n)(y_h^n)^2 p_h^n (y_h^{n+1}-y_h^n) = \\
&= \int_{\Omega} b(u_h^{n+1}) \left((y_h^n)^3-(y_h^{n+1})^3\right) p_h^n + 3\int_\Omega b(u_h^n)(y_h^n)^2 p_h^n (y_h^{n+1}-y_h^n) + \int_\Omega b(u_h^{n+1})(y_h^{n+1})^3p_h^n \\
& \qquad - \int_\Omega b(u_h^{n})(y_h^{n})^3p_h^n = \qquad \textit{(by the expansion $(y_h^{n+1})^3 = \left( y_h^n + (y_h^{n+1}-y_h^n) \right)^3$)} \\
&= 3\int_\Omega \left(b(u_h^n) - b(u_h^{n+1})\right)(y_h^n)^2 p_h^n (y_h^{n+1}-y_h^n) - 3\int_\Omega b(u_h^{n+1})(y_h^n) p_h^n (y_h^{n+1}-y_h^n)^2 \\
& \qquad - \int_\Omega b(u_h^{n+1}) p_h^n (y_h^{n+1}-y_h^n)^3 + \int_\Omega b(u_h^{n+1})(y_h^{n+1})^3p_h^n - \int_\Omega b(u_h^{n})(y_h^{n})^3p_h^n.
\end{aligned}
\]
Collecting the terms and taking advantage of the expression of the direct problem, we conclude that
\[
\begin{aligned}
RHS =& \half \norm{L^2(\partial \Omega)}{y_h^{n+1}-y_h^n}^2 + \int_{\Omega}\left( a(u_h^n) - a(u_h^{n+1})\right) \nabla (y_h^{n+1} - y_h^n) \cdot \nabla p_h^n \\
& \qquad + 3\int_\Omega \left(b(u_h^n) - b(u_h^{n+1})\right)(y_h^n)^2 p_h^n (y_h^{n+1}-y_h^n) \\
& \qquad - 3\int_\Omega b(u_h^{n+1})(y_h^n) p_h^n (y_h^{n+1}-y_h^n)^2 - \int_\Omega b(u_h^{n+1})p_h^n (y_h^{n+1}-y_h^n)^3.
\end{aligned}
\]
We now employ the Cauchy-Schwarz inequality and the regularity of the solutions of the discrete direct and adjoint problems (in particular the equivalence of the $W^{1,\infty}$ and $H^1$ norm in $V_h$: $\norm{W^{1,\infty}}{u_h} \leq C_1 \norm{H^1}{u_h}$, $C_1 = C_1(\Omega, h)$):
\[
RHS \leq C_2 \norm{L^2}{u_h^{n+1} - u_h^n}\norm{H^1}{y_h^{n+1} - y_h^n} + C_3 \norm{H^1}{y_h^{n+1} - y_h^n}^2
\]
with $C_2 = (1-k)C_1\norm{H^1}{p_h^n} + C_1 \norm{H^1}{y_h^n}\norm{H^1}{p_h^n}$ and $C_3 = 3C_1^2\norm{H_1}{y_h^n}\norm{H_1}{p_h^n} + C_1^3\norm{H_1}{p_h^n}(\norm{H_1}{y_h^n}+\norm{H_1}{y_h^{n+1}}) + \half C_{tr}^2$, being $C_{tr}$ the constant of the trace inequality in $H^1(\Omega)$.
Eventually, similarly to the computation included in the proof of Proposition \ref{prop:optcond}, one can assess that
\[
	\norm{H^1}{y_h^{n+1} - y_h^n} \leq C_4 \norm{L^2}{u_h^{n+1} - u_h^n},
\]
with $C_4 = C_4(k, C_1,\norm{H^1}{y_h^{n}},\Omega)$.
Hence, we can conclude that there exists a positive constant $\mathcal{C}_n = C_2 C_4 + C_3 C_4^2$ such that
\[
\frac{1}{\tau_n}\norm{L^2}{u_h^{n+1}-u_h^n}^2 + J(u_h^{n+1}) - J(u_h^{n}) \leq \mathcal{C}_n \norm{L^2}{u_h^{n+1} - u_h^n}^2, 
\]
and choosing $\tau_n < \mathcal{B}_n := \frac{1}{1 + \mathcal{C}_n}$ we can conclude the thesis.
\end{proof}

We are finally able to prove the following convergence result for the fully discretized Parabolic Obstacle Problem:
\begin{proposition}
Consider a starting point $u_h^0 \in \mathcal{K}_h$. Then, there exists a collection of timesteps $\{\tau_n\}$ s.t. $0 <\gamma \leq \tau_n \leq \mathcal{B}_n$ $\forall n>0$. Corresponding to $\{\tau_n\}$, the sequence $\{u_h^n\}$ generated by \eqref{eq:fulldiscrPOP} has a converging subsequence (which we still denote with $u_h^n$) such that $u_h^n \xrightarrow{W^{1,\infty}} u_h \in V_h$, which satisfies the discrete optimality conditions \eqref{eq:discrVI}.
\end{proposition}

\begin{proof}
Consider a generic collection of timesteps $\tilde{\tau}_n$ satisfying $\tilde{\tau}_n \leq \mathcal{B}_n$ $\forall n>0$. Hence, by Lemma \ref{prop:decrease}, 
\[
\sum_{n=0}^{\infty} \norm{L^2}{u_h^{n+1} - u_h^n}^2 \leq J_{\varepsilon,h}(u_h^0) \quad \text{and} \quad \sup_{n} J_{\varepsilon,h}(u_h^n) \leq J_{\varepsilon,h}(u_h^0) 
\]
which implies that $\norm{L^2}{u_h^{n+1} - u_h^n} \rightarrow 0$ and hence $u_h^n$ is bounded in $H^1(\Omega)$, and this implies that also $\{y_h^n\}$ and $\{p_h^n\}$ are bounded in $H^1(\Omega)$. According to the definition of the constants $\mathcal{C}_n$ and $\mathcal{B}_n$ reported in the proof of Lemma \ref{prop:decrease}, this entails that there exists a constant $M>0$ such that $\mathcal{C}_n \leq M$ $\forall n>0$, and equivalently there exists a positive constant $\gamma$ s.t. $\gamma \leq \mathcal{B}_n$. Hence, it is possible to choose, for each $n>0$, $\gamma \leq \tau_n \leq \mathcal{B}_n$.
\\Eventually, we conclude that there exists $u_h \in \mathcal{K}_h$ such that, up to a subsequence, $u_h^n \rightarrow u_h$ a.e. and in $W^{1,\infty}(\Omega)$ (and $y_h^n \rightarrow y_h^n := S_h(u_h^n)$, $p_h^n \rightarrow p_h$ in $H^1$ and in $W^{1,\infty}$ as well).
We exploit the expression of the discrete Parabolic Obstacle Problem \eqref{eq:fulldiscrPOP} to show that
\[
\begin{aligned}
\int_{\Omega}(1-k)\nabla y_h^n \cdot \nabla p_h^n (v_h - u_h^{n+1}) + \int_{\Omega}(y_h^n)^3 p_h^n (v_h - u_h^{n+1}) + 2 \alpha \varepsilon \int_{\Omega}\nabla u_h^{n+1} \cdot \nabla(v_h - u_h^{n+1}) \\
+ \alpha \frac{1}{\varepsilon} \int_{\Omega}(1-2u_h^n)(v_h - u_h^{n+1}) \geq - \frac{1}{\tau_n} \int_{\Omega}(u_h^{n+1}-u_h^n)(v_h-u_h^{n+1}) \quad \forall v_h \in \mathcal{K}_h,
\end{aligned}
\]
and since $\frac{1}{\tau_n} > \frac{1}{\gamma}$ $\forall n$, when taking the limit as $n \rightarrow \infty$, the right-hand side converges to $0$, which entails that $u_h$ satisfies the discrete optimality conditions \eqref{eq:discrVI}.
\end{proof}

The most remarkable outcome of the analyzed discretization of the Parabolic Obstacle Problem is the implementation of an iterative reconstruction algorithm which requires, at each iteration, the solution of two boundary value problem and of a quadratic constraint minimization problem. Indeed, introducing a basis $\{\phi_i\}_{i=1}^{N_h}$ in the discrete space $V_h$, the variational inequality in \eqref{eq:fulldiscrPOP} can be written in algebraic form. The resulting system of inequalities can be interpreted as the optimality condition of a minimization problem involving a quadratic cost functional in the compact set $[0,1]^{N_h}$, and is efficiently solved by means of the Primal-Dual Active Set method, introduced in \cite{art:bergounioux1997augemented} and applied in \cite{art:blank2013nonlocal} on a nonlocal Allen-Cahn equation. 
The final formulation of the reconstruction algorithm is the following:
\par
\begin{algorithm}[H]
\caption{Solution of the discrete Parabolic Obstacle Problem}
\label{al:POP}
\begin{algorithmic}[1]
\STATE{ Set $n = 0$ and $u_h^0 = u_0$, the initial guess for the inclusion \; }
\WHILE{$\norm{L^{\infty}(\Omega)}{u^n_h-u^{n-1}_h}>tol_{POP}$}
\STATE{solve the direct problem \eqref{eq:prob} with $u = u_h^n$\;}
\STATE{solve the adjoint problem \eqref{eq:adjoint} with $u = u_h^n$\;}
\STATE{compute $u^{n+1}$ solving \eqref{eq:fulldiscrPOP} via PDAS algorithm \;}
\STATE{update $n = n+1$;}
\ENDWHILE
\RETURN $u_h^{n}$
\end{algorithmic}
\end{algorithm}

\begin{remark}
It is a common practice to increase the performance of a reconstruction algorithm taking advantage of multiple measurements. In this context, it is possible to suppose the knowledge of $N_f$ different measurements of the electric potential on the boundary, $y_{meas,j}$ $j=1,\cdots,N_f$, associated to different source terms $f_j$. Therefore, instead of tackling the optimization of the mismatch functional $J$ as in \eqref{eq:opt}, it is possible to introduce the averaged cost functional $J^{TOT}(u) = \frac{1}{N_f}\sum_{j=1}^{N_f}J^j(u)$, where $J^j(u) = \half \norm{L^2(\partial \Omega)}{S_j(u) - y_{meas,j}}^2$, being $S_j(u)$ the solution of the direct problem \eqref{eq:prob} with source term $f = f_j$. The process of regularization, relaxation and computation of the optimality conditions is exactly the same as for $J$, and yields the same reconstruction algorithm as in Algorithm \ref{al:POP}, where at each timestep the solution of $N^f$ direct and adjoint problem must be computed.
\label{multi}
\end{remark}

\section{Numerical results}
\label{sec:numer}
In this section we report various results obtained applying Algorithm \ref{al:POP}. 
In all the numerical experiments, we consider $\Omega = (-1,1)^2$ and we introduce a shape regular tessellation $\mathcal{T}_h$ of triangles. Due to the lack of experimental measures of the boundary datum $y_{meas}$, we make use of synthetic data, i.e., we simulate the direct problem via the Finite Element method, considering the presence of an ischemic region of prescribed geometry, and extract the value on the boundary of the domain. In order to avoid to incurr an inverse crime (i.e. the performance of the reconstruction algorithm are improved by the fact that the exact data are synthetically generated with the same numerical scheme adopted in the algorithm) we introduce a more refined mesh $\mathcal{T}_h^{ex}$ on which the exact problem is solved, and interpolate the resulting datum $y_{meas}$ on the mesh $\mathcal{T}_h$. 
\par
The section is organised as follows: in Section \ref{first_results}, we describe the performance of Algorithm \ref{al:POP} for the minimization of the phase-field relaxed functional \eqref{eq:minrel}, showing that different and rather complicated geometry of inclusion can be satisfactorily reconstructed. In Section \ref{settings} we test the robustness of the reconstruction when some of the main parameters involved in the algorithm are modified. Moreover, particular attention is given to the use of a mesh-adaptative strategy.

\subsection{Parabolic Obstacle Problem: main results}
\label{first_results}
In the following test cases, we applied Algorithm \ref{al:POP} in order to reconstruct inclusions of different geometries, in order to investigate the effectiveness of the introduced strategy. We used the same computational mesh $\mathcal{T}_h$ (mesh size $h_{max} = 0.04$, nearly $6000$ elements) for the numerical solution of the boundary value problems involved in the procedure, whereas the mesh $\mathcal{T}_h^{ex}$ for the generation of each different synthetic datum associated to the different inclusions is specifically refined along the boundary of the respective ischemic region. According to Remark \ref{multi}, we make use of $N_f = 2$ different measurements, associated to the source terms $f_1(x,y) = x$ and $f_2(x,y) = y$. The main parameters for all the simulations lie in the ranges reported in Table \ref{tab:param}. We make use of the same relationship between $\varepsilon$ and $\tau$ as in \cite{art:deckelnick}.
\begin{table}[h!]
\centering
\begin{tabular}{c|c|c|c}
$\alpha$ & $\varepsilon$ & $\tau$ &  $tol_{POP}$ \\
\hline
$10^{-4} \div 10^{-3}$ & $1/(8 \pi)$ & $(0.01 \div 0.1)/\varepsilon$ &  $10^{-4}$
\end{tabular}
\caption{Range of the main parameters}
\label{tab:param}
\end{table}
The initial guess for each simulation is $u_0 \equiv 0$.
\par In Figure \ref{fig:circle} we report some of the iterations of Algorithm \ref{al:POP} for the reconstruction of a circular inclusion ($\alpha = 0.0001$, $\tau = 0.01/\varepsilon$). The boundary $\partial \omega$ is marked with a black line, which is superimposed to the contour plot of the approximation of the indicator function $u_h^n$ at different timesteps $n$. The algorithm converged after $N_{tot} = 568$ iterations, corresponding to a final (fictitious) time $T_{tot}= 1427.54$.  
\begin{figure}[h!]
			\centering
			\subfloat[$n = 30$]{
		    	\includegraphics[width=0.33\textwidth]{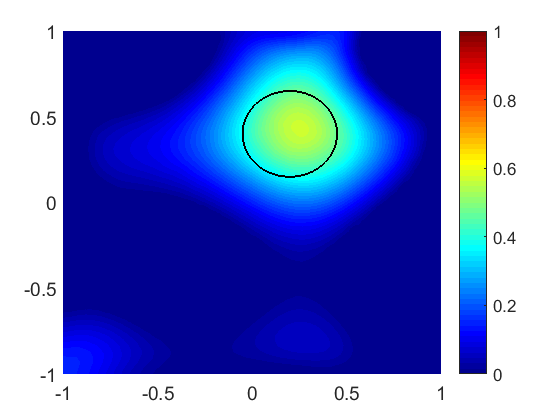}
			}
			\subfloat[$n = 90$]{
		    	\includegraphics[width=0.33\textwidth]{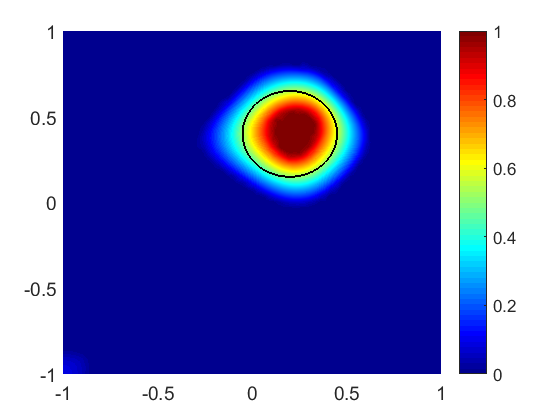}
			}
			\subfloat[$n = 568$]{
		    	\includegraphics[width=0.33\textwidth]{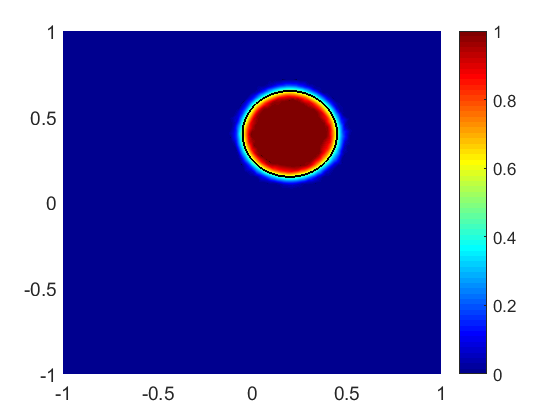}
			}
	\caption{Reconstruction of a circular inclusion: successive iterations}
	\label{fig:circle}
\end{figure}
In Figure \ref{fig:many} we investigate the ability of the algorithm to reconstruct inclusions of rather complicated geometry. For each test case, we show the contour plot of the final iteration of the reconstruction (the total number of iterations $N$ and the final time $T$ are reported in the caption), and the boundary of the exact inclusion is overlaid in black line. Moreover, each result is equipped with the graphic (in semilogarithmic scale) of the evolution of the cost functional $J_\varepsilon$, split into the components $J_{PDE}(u) = \half \norm{L^2(\partial \Omega)}{S(u)-y_{meas}}^2$ and $J_{regularization}(u) = \alpha \varepsilon \norm{L^2(\Omega)}{\nabla u}^2 + \frac{\alpha}{\varepsilon} \int_{\Omega}u(1-u)$.
\begin{figure}[h!]
			\centering
			\subfloat[$N_{tot} = 1500$, $T_{tot} = 753.98$, \protect\linebreak $\alpha = 0.001$, $\tau = 0.02/\varepsilon$]{
		    	\includegraphics[width=0.33\textwidth]{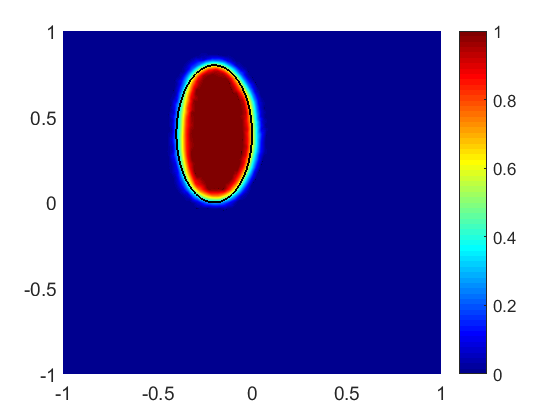}
			}
			\subfloat[$N_{tot} = 1272$, $T_{tot} = 639.38$, \protect\linebreak $\alpha = 0.0001$, $\tau = 0.02/\varepsilon$]{
		    	\includegraphics[width=0.33\textwidth]{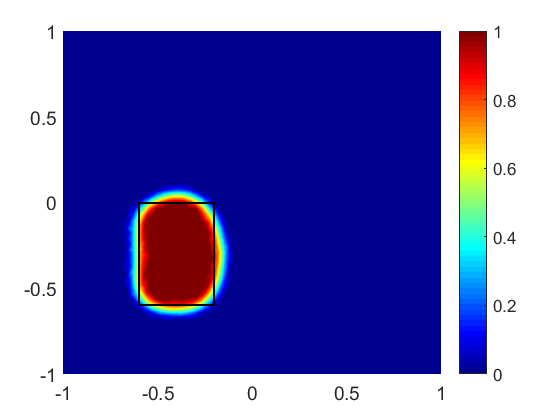}
			}
			\subfloat[$N_{tot} = 4670$, $T_{tot} = 2347.40$,\protect\linebreak $\alpha = 0.0001$, $\tau = 0.02/\varepsilon$]{
		    	\includegraphics[width=0.33\textwidth]{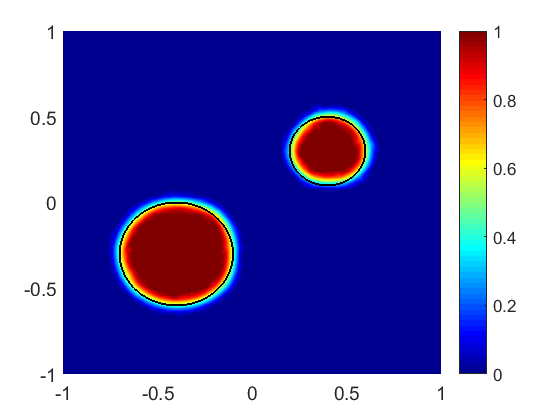}
			}
			\\
			\subfloat[Ellipse: evolution of $J_\varepsilon$]{
		    	\includegraphics[width=0.33\textwidth]{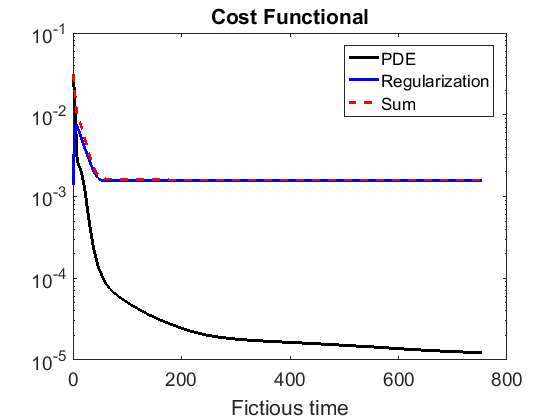}
			}
			\subfloat[Rectangle: evolution of $J_\varepsilon$]{
		    	\includegraphics[width=0.33\textwidth]{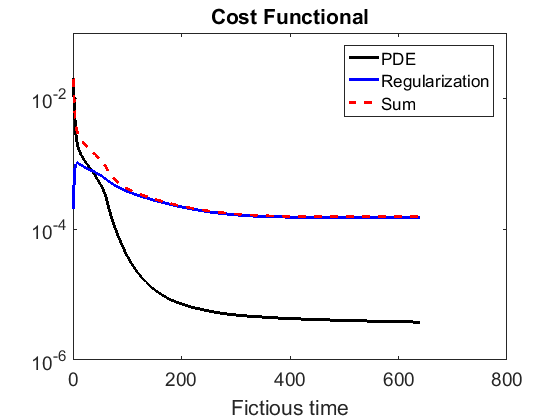}
			}
			\subfloat[Two circles: evolution of $J_\varepsilon$]{
		    	\includegraphics[width=0.33\textwidth]{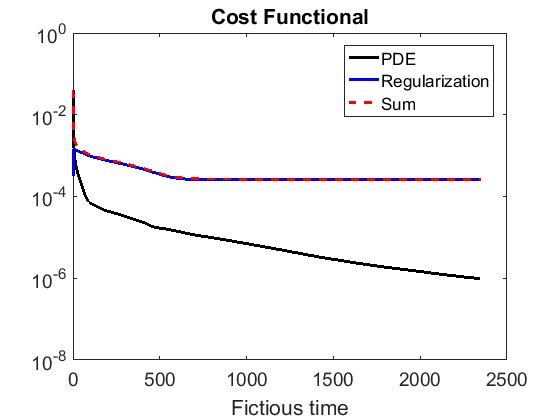}
			}
	\caption{Reconstruction of various inclusions}
	\label{fig:many}
\end{figure}
The reported results consist in approximations of minimizers of $J_\varepsilon$ in $\mathcal{K}$: they are smooth function and range between $0$ and $1$. They show large regions in which they attain the limit values $0$ and $1$, and a region of \textit{diffuse interface} between them, whose thickness is about $\varepsilon/2$. As Figures \ref{fig:circle} and \ref{fig:many} show, the algorithm is capable of reconstructing inclusion of rather complicated geometry. The identification of smooth inclusion is performed with higher precision, whereas it seems that the accuracy is low in presence of sharp corners. We point out that we don't need to have any \textit{a priori} knowledge on the topology of the inclusion $\omega$, i.e., the number of connected components is correctly identified. 

\subsection{Parabolic Obstacle Problem: setting of the parameters}
\label{settings}
This section is devoted to the description of the performance of Algorithm \ref{al:POP} when some of the main parameters and settings are perturbed. 
\par In particular, we start assessing that the final result of the reconstruction is independent of the initial guess imposed as a starting point of the Parabolic Obstacle problem. In Figure \ref{fig:starting} we compare the behaviour of the algorithm applied to the reconstruction of a circular inclusion (the same as in Figure \ref{fig:circle}), where we impose a different initial datum with respect to the constant zero function. In the first experiment, we start from an initial datum which is the indicator function of an arbitrarily chosen region. In the second one, we impose as a starting point the indicator function of a sublevel of the \textit{topological gradient} of the cost functional $J$. As investigated in \cite{art:BMR}, the topological gradient is a powerful tool for the detection of small-size inclusions, which yield a small perturbation in the cost functional with respect to the background (unperturbed) case. The position of a small inclusion is easily identified by searching for the point where the topological gradient of $J$ attains its (negative) minimum. As the information given by the topological gradient $G$ has shown to be remarkable even in the case of large-size inclusions (see, e.g., \cite{art:BCCMR}, \cite{art:rapun}), we take advantage of it by computing $G$ (see Theorem 3.1 in \cite{art:BMR}), setting a threshold $G_{thr}$ and defining $u_0 = \chi_{\{G \leq G_{thr}\}}$. 
\begin{figure}[h!]
			\centering
			\subfloat[Initial guess: arbitrary]{
		    	\includegraphics[width=0.33\textwidth]{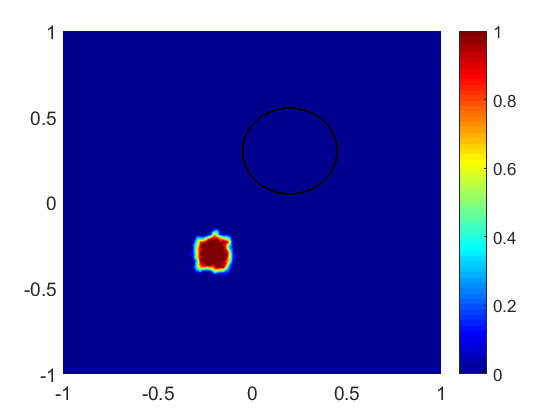}
			}
			\subfloat[Intermediate: $n = 60$]{
		    	\includegraphics[width=0.33\textwidth]{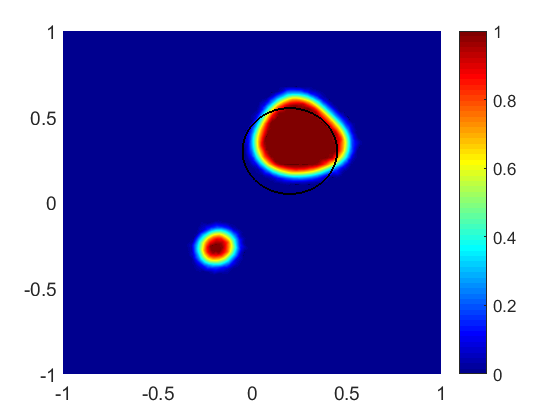}
			}
			\subfloat[Final: $N_{tot} = 661$, $T_{tot} = 1661.27$]{
		    	\includegraphics[width=0.33\textwidth]{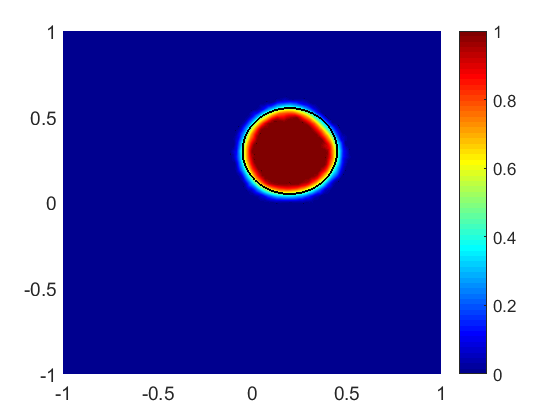}
			}
			\\
			\subfloat[Initial guess: topological]{
		    	\includegraphics[width=0.33\textwidth]{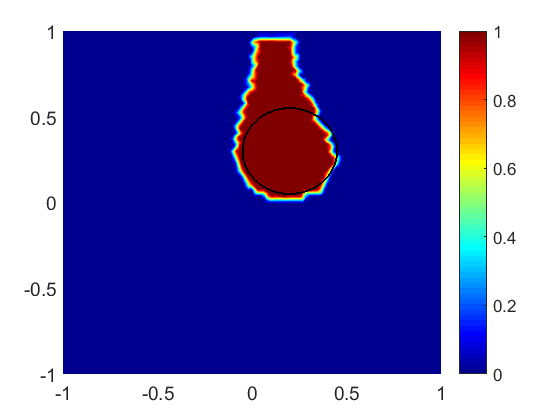}
			}
			\subfloat[Intermediate: $n = 50$]{
		    	\includegraphics[width=0.33\textwidth]{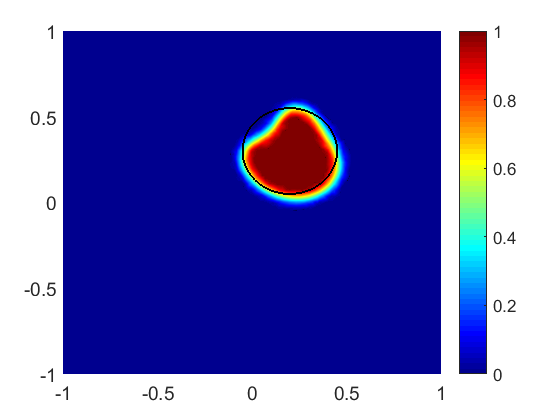}
			}
			\subfloat[Final: $N_{tot} = 489$, $T_{tot} = 1228.99$]{
		    	\includegraphics[width=0.33\textwidth]{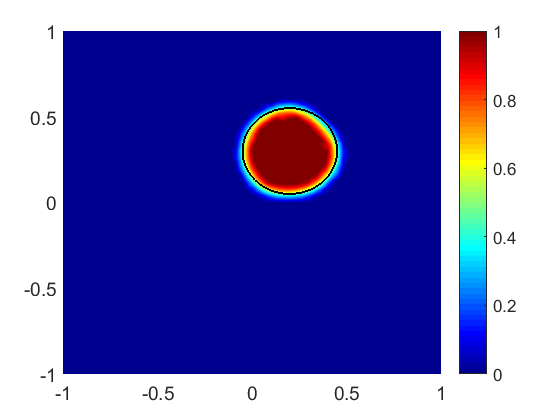}
			}
	\caption{Reconstruction of a circular inclusion with different initial conditions}
	\label{fig:starting}
	\end{figure}
The results reported in Figure \ref{fig:starting} show the starting point of the algorithm, an intermediate iteration and the final reconstruction. In both cases we set $\alpha = 0.001$, $\varepsilon = 1/(8 \pi)$ and $\tau = 0.1/\varepsilon$.We underline that the result in each case is similar to the one depicted in Figure \ref{fig:circle}, but through the second strategy it was possible to perform a smaller number of iterations. 
\par
Another interesting investigation is the comparison of the results obtained when perturbing the relaxation parameter $\varepsilon$. In Figure \ref{fig:eps} we report the final reconstruction of an ellipse-shaped inclusion when setting $\varepsilon = \frac{1}{4\pi}, \frac{1}{8\pi}, \frac{1}{8\pi}$.
\begin{figure}[h!]
			\centering
			\subfloat[$\varepsilon = \frac{1}{4\pi}$: $N_{tot} = 358$, $T_{tot} = 224.94$]{
		    	\includegraphics[width=0.33\textwidth]{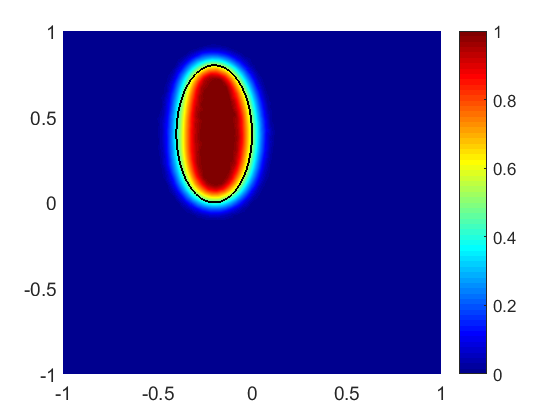}
			}
			\subfloat[$\varepsilon = \frac{1}{8\pi}$: $N_{tot} = 1500$, $T_{tot} = 753.98$]{
		    	\includegraphics[width=0.33\textwidth]{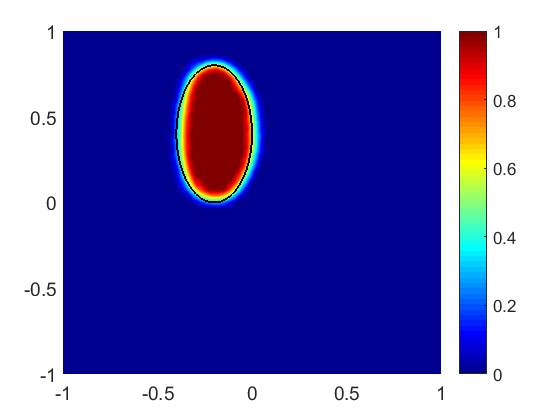}
			}
			\subfloat[$\varepsilon = \frac{1}{16\pi}$: $N_{tot} = 3514$, $T_{tot} = 1766.33$]{
		    	\includegraphics[width=0.33\textwidth]{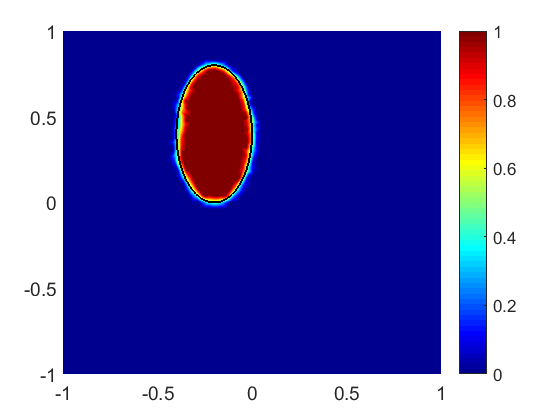}
			}
			\caption{Reconstruction of a circular inclusion with different $\varepsilon$}
	\label{fig:eps}
\end{figure}
As expected, it is possible to remark that the thickness of the diffuse interface region decreases as $\varepsilon$ decreases. Nevertheless, one must take into account the size of the computational mesh $\mathcal{T}_h$: in the last test, the thickness of the region in which the final iteration $u_h^{N_tot}$ increases from $0$ to $1$ is of the same order of magnitude as $h_{max}$. This is rather likely the reason why the edge of the reconstructed inclusion appears to be irregular and jagged. A natural strategy to avoid the problem would be to make use of a finer mesh, e.g., ensuring that $h_{max}<\varepsilon/10$; however, that could result in an extremely high computational effort. It is possible to overcome this drawback by introducing an adaptive mesh refinement strategy, i.e., by locally refining the mesh close to the region of the detected edges. In Figure \ref{fig:adapt} we compare the result obtained when approximating a rectangular and a circular inclusion with $\varepsilon = \frac{1}{16\pi}$ on the reference mesh or through a process of mesh adaptation. We invoked a goal-oriented mesh adaptation algorithm each $N_{adapt} = 50$ iterations, requiring for a higher refinement of the grid in proximity of higher values of $|\nabla u_h^n|$ and for a lower refinement in the regions where $u_h^n$ is approximatively constant. This allows to have more precise reconstruction even for small $\varepsilon$, almost without increasing the global number of elements of the mesh. In Figure \ref{fig:adapt}, we also report the final configuration of the refined computational mesh.  
\begin{figure}[h!]
			\centering
			\subfloat[No adaptation, $N_{tot} = 2442$, $T_{tot} = 2454.97$]{
		    	\includegraphics[width=0.33\textwidth]{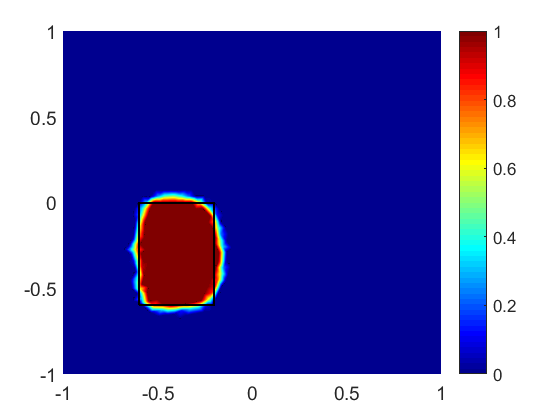}
			}
			\subfloat[Adaptation, $N_{tot} = 2189$, $T_{tot} = 2200.62$]{
		    	\includegraphics[width=0.33\textwidth]{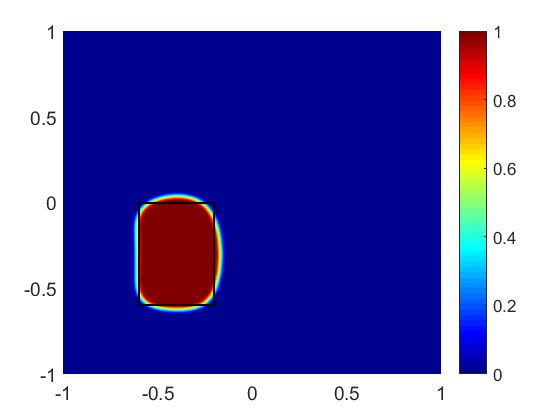}
			}
			\subfloat[Final adapted mesh]{
		    	\includegraphics[width=0.33\textwidth]{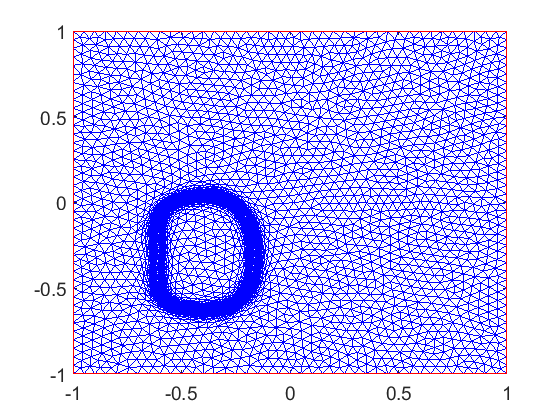}
			}
			\\
			\subfloat[No adaptation, $N_{tot} = 2210$, $T_{tot} = 2221.73$]{
		    	\includegraphics[width=0.33\textwidth]{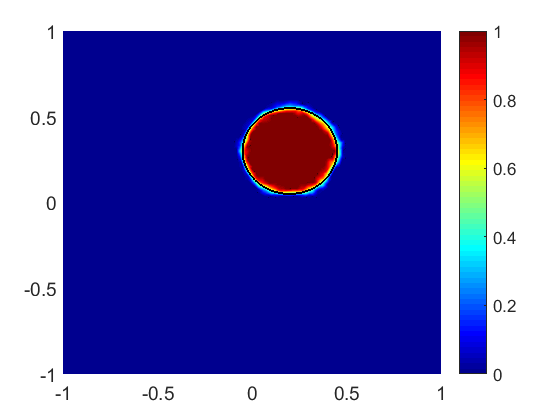}
			}
			\subfloat[Adaptation, $N_{tot} = 2306$, $T_{tot} = 2318.24$]{
		    	\includegraphics[width=0.33\textwidth]{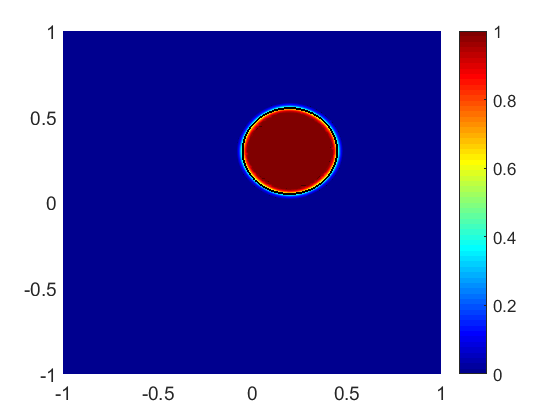}
			}
			\subfloat[Final adapted mesh]{
		    	\includegraphics[width=0.33\textwidth]{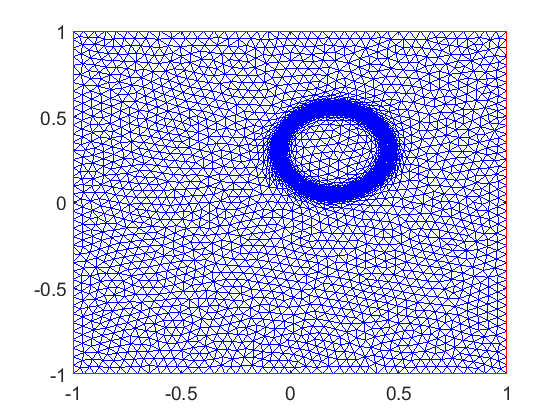}
			}
	\caption{Mesh adaptation: result comparison}
	\label{fig:adapt}
\end{figure}

\section{Comparison with the Shape Derivative approach}
\label{sec:shape}
In the previous sections, we have analyzed in detail the phase-field relaxation of the minimization problem expressed in \eqref{eq:minreg}. We now aim at describing the relationship between this method and the Shape Derivative approach, which has become a very common strategy to tackle the reconstruction of discontinuous coefficients. The algorithm based on the shape derivative consists in updating the shape of the inclusion to be reconstructed by perturbing its boundary along the directions of the vector field which causes the greatest descent of the cost functional, that can be deduced by computing the shape derivative of the functional itself. In this section, we first theoretically investigate the relationship between the shape derivative of the cost functional $J_{reg}$ and the Fréchet derivative of $\Je$ and then report a comparison between the numerical results of the two algorithms in a set of benchmark cases.

\subsection{Sharp interface limit of the Optimality Conditions}
\label{sec:sharplim}
In order to study the relationship between the optimality conditions in the phase-field approach and the ones derived in the sharp case, we follow an analogous approach as in \cite{art:garcke2016}. First of all, in Proposition \ref{prop:sharpOC} we introduce the necessary optimality condition for the sharp problem \eqref{eq:minreg}, taking advantage of the computation of the material derivative of the cost functional. We then define in Proposition \ref{prop:pfOC} similar optimality conditions for the relaxed problem \eqref{eq:minrel}, which are related but not equivalent to the one stated in \eqref{eq:OC}-\eqref{eq:VI} through the Fréchet derivative. In Proposition \ref{prop:SIL} we finally assess the convergence of the phase-field optimality condition to the sharp one when $\varepsilon \rightarrow 0$.
\par
For the sake of simplicity, in this section we will refer to $J_{reg}$ as $J$. Consider the minimization problem (as in \eqref{eq:minreg}):
\begin{equation}
	\argmin_{u \in X_{0,1}} J(u); \quad J(u) = \half \norm{L^2(\partial\Omega)}{S(u) - y_{meas}}^2 + \alpha TV(u).
\label{eq:minreg2}
\end{equation}
Since $u \in X_{0,1}$ implies that $u = \chi_\omega$, being $\omega$ a finite-perimeter subset of $\Omega$, we can perturb $u$ by means of a vector field $\phi_t: \Omega \rightarrow \R^2$, $\phi_t(x) = x + tV(x)$, being 
\begin{equation}
V \in C^1(\Omega) \text{ s.t. }  V(x) = 0 \text{ in }\Omega^{d_0} = \{x \in \Omega \ s.t. \ dist(x,\partial \Omega) \leq d_0\}.
\label{eq:V}
\end{equation}
Consider the family of functions $\{u_t\}$: $u_t = u \circ \phi_t^{-1}$: we can compute the shape derivative of the functional $J$ in $u$ along the direction $V$ (see \cite{book:DZ2011}) as
\begin{equation}
	DJ(u)[V] := \lim_{t \rightarrow 0} \frac{J(u_t)-J(u)}{t},
\label{eq:shapedersharpDef}
\end{equation}  
where $J(u_t)$ is the cost functional evaluated in the deformed domain $\Omega_t=\phi_t(\Omega)$ but, according to \eqref{eq:V}, $\Omega_t$ and $\Omega$ are the same set, thus we do not adopt a different notation.
We prove the following result:
\begin{proposition}
If $u$ is a solution of \eqref{eq:minreg2} and $f \in L^2(\Omega)$ satisfies the hypotheses in Proposition \ref{prop:geq} or \ref{prop:geq2}, then
\begin{equation}
DJ(u)[V] = 0 \qquad \textit{ for all the smooth vector fields $V$},
\label{eq:OCmatder}
\end{equation}
The shape derivative is given by:
\begin{equation}
	DJ(u)[V] = \int_{\partial \Omega}(S(u) - y_{meas})\dot{S}(u)[V] + \int_\Omega(div V - DV \normal \cdot \normal)d|Du|,
\label{eq:matdersharp}
\end{equation}
where $d|Du|=\delta_{\partial \omega} dx$, $\normal$ is the generalized unit normal vector (see \cite{book:giusti}) and $\dot{S}(u)[V] =: \dot{S}$, the material derivative of the solution map, solves
\begin{equation}
\begin{aligned}
\int_\Omega a(u) \nabla \dot{S} \cdot \nabla v + \int_\Omega b(u)3 S(u)^2 \dot{S} v = &-\int_\Omega a(u)\mathcal A \nabla S(u) \cdot \nabla v -\int_\Omega b(u)S(u)^3 v div V + \\
&\int_\Omega div(fV)v  \qquad \forall v \in H^1(\Omega),
\label{eq:matder}
\end{aligned}
\end{equation}
being $\mathcal{A} = div V - (DV + DV^T)$. 
\label{prop:sharpOC}
\end{proposition}
\begin{proof}
We start by deriving the formula of the material derivative of the solution map. 
Define $S_0 = S(u)$ and $S_t: \Omega \rightarrow \R$, $S_t = S(u_t) \circ \phi_t$.
Then, applying the change of variables induced by the map $\phi_t$, it holds that
\begin{equation}
	\int_{\Omega} a(u) A(t) \nabla S_t \cdot \nabla v + \int_{\Omega} b(u)S_t^3 v |det D\phi_t| = \int_{\Omega} (f\cdot \phi_t) v |det D\phi_t| \qquad \forall v \in H^1(\Omega),
\label{eq:2}
\end{equation}
where $A(t) = D\phi_t^{-T} D\phi_t^{-1} |det D\phi_t|$.
By computation,
\[
\frac{d}{dt} A(t) = \mathcal{A} = (div V)I - (DV^t+DV) \qquad \text{and} \qquad \frac{d}{dt} |det D\phi_t| = div V.
\] 
Subtract \eqref{eq:prob} from \eqref{eq:2} and divide by $t$: then $w_t = \frac{S_t - S_0}{t}$ is the solution of
\begin{equation}
\begin{aligned}
  \int_\Omega a(u) A(t) \nabla w_t \cdot \nabla v &+ \int_\Omega b(u)q_t w_t v |det(D\phi_t)| = -\int_{\Omega}a(u) \frac{A(t)-I}{t} \nabla S_0 \cdot \nabla v \\ &- \int_{\Omega} \frac{|det(D\phi_t)|-1}{t}b(u)S_0^3v + \int_{\Omega} \frac{1}{t}(f\circ \phi_t) v |det(D\phi_t)|  -\int_{\Omega} \frac{1}{t} fv \qquad \forall v \in H^1(\Omega),
\end{aligned}
\label{eq:4}
\end{equation}
where the norm of the right-hand side in the dual space of $H^1(\Omega)$ is bounded by
	\[
	\begin{aligned}
		\norm{L^{\infty}(\Omega)}{\frac{A-I}{t}}&\norm{H^1(\Omega)}{S_0} + \norm{L^{\infty}(\Omega)}{\frac{|det(D\phi_t)|-1}{t}}\norm{H^1(\Omega)}{S_0} \\&+ \norm{L^{\infty}(\Omega)}{\frac{|det(D\phi_t)|-1}{t}}\norm{L^2(\Omega)}{f} + C(\norm{C(\Omega)}{V})\norm{H^1(\Omega)}{f} \leq C_F,
	\end{aligned}
	\]
	being $C_F$ independent of $t$. Moreover, the matrix $A(t)$ is symmetric positive definite: $(A(t) y) \cdot y\geq \half \norm{}{y}^2$ $\forall y \in \R^2, \forall t$. Together with the property that $q_t = u_t^2 + u_t u + u^2 \geq \frac{3}{4}u^2$, and thanks to Proposition \ref{prop:geq} and to the Poincaré inequality in Lemma \ref{Lemma1},  
\[
	\norm{H^1}{w_t}^2 \leq C k\norm{L^2}{\nabla w_t}^2 +\frac{3}{4}Q \norm{L^2(\Omega^*)}{w_t}^2 \leq C_F\norm{H^1}{w_t}.
\]
Thus, $\norm{H^1}{w_t}$ is bounded independently of $t$, from which it follows that $\norm{H^1(\Omega)}{S_t - S_0} \leq C t$ and that every sequence $ \{w_n\} = \left\{ w_{t_n}, \  t_n \rightarrow 0 \right\}$ is bounded in $H^1(\Omega)$, thus $w_t \xrightharpoonup{H^1} w \in H^1(\Omega)$. We aim at proving that $w$ is also the limit of $w_t$ in the strong convergence, which entails that
\[
	\dot{S}(u)[V]:= \lim_{t\rightarrow 0} \frac{S_t - S_0}{t} = w.
\]
First of all, we show that $w$ is the solution of problem \eqref{eq:matder}. It follows from \eqref{eq:4}, since $q_t w_t = \frac{1}{t}(S_t^3 - S_0^3) = \frac{1}{t}((S_0 +t w_t)^3 - S_0^3) = 3 S_0^2 w_t + 3t S_0 w_t^2 + t^2 w_t^3$, that
\begin{equation}
\begin{aligned}
 	\int_{\Omega} a(u) A(t) \nabla w_t \cdot \nabla v &+ \int_{\Omega} b(u) 3 S_0^2 w_t v |det D\phi_t| = -\int_{\Omega}a(u) \frac{A(t)-I}{t} \nabla S_0 \cdot \nabla v \\ &-\int_{\Omega} \frac{|det D\phi_t|-1}{t} b(u)S_0^3v - \int_{\Omega} b(u) 3t S_0 w_t^2  v |det D\phi_t| - \int_{\Omega} b(u) t^2 w_t^3 v |det D\phi_t|\\ &+ \int_{\Omega} (f \circ \phi_t)\frac{|det D\phi_t|-1}{t} v -\int_{\Omega}\frac{(f\circ \phi_t) - f}{t}v  \qquad \forall v \in H^1(\Omega).
\end{aligned}
\label{eq:4.5}
\end{equation}
Taking the limit as $t \rightarrow 0$ and by the weak convergence of $w_t$ in $H^1$, we recover the same expression as in \eqref{eq:matder}. One may eventually show that $w_t\xrightarrow{H^1} w$. In order to do this we start proving that
\begin{equation}
	\int_\Omega a(u) A(t)|\nabla w_t|^2 + \int_\Omega b(u)|det D\phi_t|3S_0^2w_t^2 \rightarrow \int_\Omega a(u) |\nabla w|^2 + \int_\Omega b(u)3S_0^2w^2.
\label{eq:6}
\end{equation}
Indeed, take \eqref{eq:4.5} and substitute $v = w_t$: using the weak convergence of $w_t$ in the right-hand side, we obtain that 
\[
\begin{aligned}
\int_\Omega a(u) A(t)|\nabla w_t|^2 &+ \int_\Omega b(u)|det D\phi_t|3S_0^2w_t^2 \rightarrow  -\int_{\Omega}a(u) \mathcal{A} \nabla S_0 \cdot \nabla w -\int_{\Omega} divV \ b(u)S_0^3w \\
	&+ \int_{\Omega} f w\ divV - \int_{\Omega}\nabla f \cdot V w  \overset{\eqref{eq:matder}}{=} \int_{\Omega} a(u) |\nabla w|^2 + \int_{\Omega} b(u) 3 S_0^2 w^2. 
\end{aligned}
\]
We then compute:
\begin{equation}
\begin{aligned}
\int_\Omega a(u)& A(t) |\nabla(w_t - w)|^2 + \int_\Omega b(u) 3 S_0^2 (w_t - w)^2 |det D\phi_t| =
\\ & \int_\Omega a(u) A(t) |\nabla w_t|^2 + \int_\Omega a(u) A(t) |\nabla w|^2 - 2 \int_\Omega a(u) A(t) \nabla w_t \cdot \nabla w 
\\ & +\int_\Omega b(u)3S_0^2 w_t^2 |det D\phi_t| + \int_\Omega b(u) 3S_0^2 w^2  |det D\phi_t|- 2\int_\Omega b(u)3S_0^2w_t w |det D\phi_t|.
\end{aligned}
\label{eq:eq7}
\end{equation}
Using \eqref{eq:6}, the convergence of $A$ to $I$ and of $|det D\phi_t|$ to $1$, and the fact that $w_t \xrightharpoonup{H^1} w$, we derive that
\[
\int_\Omega a(u) |\nabla(w_t - w)|^2 + \int_\Omega b(u) 3S_0^2(w_t - w)^2 \rightarrow 0
\]
A combination of the Proposition \ref{prop:geq} and of the Poincarè inequality in Lemma \ref{Lemma1} allows to conclude that also $\norm{H^1}{w_t - w} \rightarrow 0$.
\par
We now prove the necessary optimality conditions for the optimization problem \eqref{eq:minreg2}.
The derivative of the quadratic part of the cost functional $J$ can be easily computed by means of the material derivative of the solution map:
\begin{equation}
\begin{aligned}
	\lim_{t \rightarrow 0} \half &\int_{\partial \Omega}\frac{(S(u_t) - y_{meas})^2 |det(D\phi_t)| - (S_0-y_{meas})^2}{t} \qquad \textit{(since $S(u_t) = S_t$ on $\partial \Omega$)}\\
	&= \lim_{t \rightarrow 0} \half \int_{\partial \Omega}(S_t - y_{meas})^2 \frac{|det(D\phi_t)| - 1}{t} +  \lim_{t \rightarrow 0} \half \int_{\partial \Omega}\frac{(S_t - y_{meas})^2 - (S_0-y_{meas})^2}{t}
 \\&= \half \int_{\partial \Omega}(S_0-y_{meas})^2 div V + \int_{\partial \Omega} \dot{S}(u)[V](S_0-y_{meas}),
\end{aligned}
\label{eq:7}
\end{equation}
and the first integral in the latter expression vanishes since $V=0$ on $\Omega_{d_0}$.
On the other hand, using Lemma 10.1 of \cite{book:giusti} and the remark 10.2, we recover the expression for the derivative of the Total Variation of $u$, which is the same reported in \eqref{eq:matdersharp}.
\end{proof}

The optimality conditions reported in \eqref{eq:OCmatder} are, at the best of our knowledge, the most general result which can be obtained in this case, i.e. by simply assuming that $u =\chi_\omega$ and $\omega$ is a set of finite perimeter. We point out that, assuming more \textit{a priori} knowledge on the $u$, it is possible to recover from \eqref{eq:matdersharp} the expression of the \textit{shape derivative} of the cost functional $J$. The following proposition can be rigorously proved by means of an analogous argument as in \cite{art:abfkl}, except for the derivative of the perimeter penalization, which can be found in Section 9.4.3 in \cite{book:DZ2011}.
\begin{proposition}
Suppose that $\omega \subset \Omega$ is open, connected, well separated from the boundary $\partial \Omega$ and regular (at least of class $C^2$), and $u = \chi_\omega$. Then, the expression of the shape derivative of the cost functional $J$ along a smooth vector field $V$ is:
\begin{equation}
	DJ(u)[V] = \int_{\partial \omega} \left[ (1-k)\left( \nabla_{\tau}S(u) \cdot \nabla_\tau w  + \frac{1}{k} \nabla_\normal S(u)^e\cdot \nabla_\normal w^e \right) + S(u)^3 w  + h \right] V\cdot \normal \qquad \forall V,
\label{eq:shapeder}
\end{equation}
where $w$ is the solution of the adjoint problem (see \eqref{eq:adjoint}). The gradients $\nabla S(u)$ and $\nabla w$ are decomposed in the normal and tangential component with respect to the boundary $\partial \omega$, and due to the transmission condition of the direct problem their normal components are discontinuos across $\partial \omega$: the valued assumed in $\Omega \setminus \omega$ is marked as $\nabla_\normal S(u)^e$. The term $h$ is instead the mean curvature of the boundary.
\label{prop:shapeder}
\end{proposition}
For the sake of completeness, we point out that the latter result can be easily generalized to the case in which $\omega$ is the union of $N_c$ disjoint, well separated, components, each of them satisfying the expressed hypotheses. Thanks to the results recently obtained in \cite{art:bfv}, we expect formula \eqref{eq:shapeder} to be valid also under milder assumption, in particular for polygons.
%
\par
We aim at demonstrating that the expression of the shape derivative reported in \eqref{eq:OCmatder} is the limit, as $\varepsilon \rightarrow 0$, of a suitable derivative of the relaxed cost functional $\Je$. In order to accomplish this result, we need to introduce necessary optimality conditions for the relaxed problem $\eqref{eq:minrel}$ which are different from the ones reported in Proposition \ref{prop:optcond} and can be derived by the same technique as in Proposition \ref{prop:sharpOC} as shown in the following result.
\begin{proposition}
If $u_\varepsilon$ is a solution of \eqref{eq:minrel}, then
\begin{equation}
DJ_\varepsilon(u_\varepsilon)[V] = 0 \qquad \textit{ for all the smooth vector fields $V$},
\label{eq:pfOC}
\end{equation}
The expression of the derivative is given by:
\begin{equation}
\begin{aligned}
DJ_\varepsilon(u_\varepsilon)[V] = &\int_{\partial \Omega}(S(u_\varepsilon) - y_{meas})\dot{S}(u_\varepsilon)[V] + \alpha \varepsilon \int_\Omega |\nabla \ue|^2 div V \\ &- 2\alpha \varepsilon \int_\Omega DV \nabla \ue \cdot \nabla \ue + \frac{\alpha}{\varepsilon} \int_{\Omega} \ue(1-\ue)div V
\end{aligned}
\label{eq:shapederpf}
\end{equation}
where $\dot{S}(u_\varepsilon)[V]$ solves the same problem as in \eqref{eq:matder}, replacing $u$ with $u_\varepsilon$.
\label{prop:pfOC}
\end{proposition}
\begin{proof}
The same strategy as in the proof of Proposition \ref{prop:sharpOC} can be adapted to compute $\dot{S}(\ue)[V]$ and the derivative of the first term of the cost functional. We now derive with the same computational rules the relaxed penalization term. Recall 
\[
	F_\varepsilon(\ue) = \alpha \varepsilon \int_{\Omega}|\nabla \ue|^2 + \frac{\alpha}{\varepsilon} \int_\Omega \psi(\ue),
\]
being $\psi: \R \rightarrow \R,$ $\psi(x) = x(1-x)$. After the deformation from $u_\varepsilon$ to $u_\varepsilon \circ \phi_t^{-1}$ and applying the change of variables induced by $\phi_t$,
\[
	F_\varepsilon(\ue \circ \phi_t^{-1}) = \alpha \varepsilon \int_{\Omega} A(t) \nabla \ue \cdot \nabla \ue + \frac{\alpha}{\varepsilon} \int_{\Omega} \psi \circ \ue \circ \phi_t^{-1}. 
\] 
Hence,
\[
\begin{aligned}
\dot{F}_\varepsilon(u_\varepsilon)[V] &= \lim_{t \rightarrow 0}\frac{F_\varepsilon(\ue \circ \phi_t^{-1}) - F_\varepsilon(\ue)}{t} = \alpha \varepsilon \int_{\Omega} \mathcal{A}\nabla \ue \cdot \nabla \ue + \alpha \varepsilon \frac{\alpha}{\varepsilon}\int_{\Omega} \psi(\ue)div V =  \\
&= \alpha \varepsilon\int_\Omega |\nabla \ue|^2 div V - \alpha \varepsilon \int_\Omega (DV + DV^T)\nabla \ue \cdot \nabla \ue + \frac{\alpha}{\varepsilon}\int_\Omega \ue(1-\ue)div V,
\end{aligned}
\] 
which is the same expression as in \eqref{eq:shapederpf}, since $DV^T \nabla v \cdot \nabla v = DV \nabla v \cdot \nabla v$.
\end{proof}

We point out that the optimality conditions deduced in the latter proposition are not equivalent to the ones expressed in Proposition \ref{prop:optcond} via the Fréchet derivative of $\Je$. Nevertheless, if $u_\varepsilon$ satisfies \eqref{eq:OC}-\eqref{eq:VI}, then it also satisfies \eqref{eq:pfOC} (it is sufficient to consider in \eqref{eq:OC} $v = \ue \circ \phi_t^{-1}$, which belongs to $\mathcal{K}$ thanks to the regularity of $V$), whereas the contrary is not valid in general. In particular, due to the regularity of the perturbation fields $V$, the optimality conditions \eqref{eq:pfOC} do not take into account possible topological changes of the inclusion: for example, the number of connected components of $\omega$ cannot change. We remark that this holds also for the optimality conditions \eqref{eq:OCmatder} for the sharp problem, and consists in a limitation for the effectiveness of the reconstruction via a shape derivative approach: the initial guess of the reconstruction algorithm and the exact inclusion must be diffeomorphic.
\par
We are now able to show the sharp interface limit of the expression of the shape derivative of the relaxed cost functional $\Je$ as $\varepsilon \rightarrow 0$, which is done in the following proposition.
\begin{proposition}
Consider a family ${\bar{u}_\varepsilon}$ s.t. $\bar{u}_\varepsilon \in \mathcal{K}$ $\forall \varepsilon > 0$ and $\bar{u}_\varepsilon \xrightarrow{L^1} \bar{u} \in BV(\Omega)$  as $\varepsilon \rightarrow 0$. Then, 
\[
DJ_\varepsilon(\bar{u}_\varepsilon)[V] \rightarrow DJ(\bar{u})[V] \qquad \textit{for every smooth vector field $V$.}
\]
\label{prop:SIL} 
\end{proposition}
\begin{proof}
We follow a similar argument as in the proof of \cite[Theorem 21]{art:garcke2016}. Thanks to Proposition \ref{prop:continuity}, $\bar{\ue} \xrightarrow{L^1} \bar{u}$ $\Rightarrow$ $S(\bar{\ue}) \xrightarrow{H^1} S(\bar{u})$. Also $\dot{S}(\bar{\ue})[V] \xrightarrow{H^1} \dot{S}(\bar{u})[V]$: the proof is done by subtracting the equations of which $\dot{S}(\bar{\ue})[V]$ and $\dot{S}(\bar{u})[V]$ and verifying that the norm of their difference is controlled by the norm of $S(\bar{\ue})-S(\bar{u})$ in $H^1(\Omega)$. Thanks to these results, surely 
\[
\int_\Omega (S(\ue)-y_{meas})\dot{S}(\bar{\ue})[V] \rightarrow \int_\Omega (S(u) -y_{meas})\dot{S}(\bar{u})[V].
\]
Eventually, the convergence
\[
\alpha \varepsilon \int_\Omega |\nabla \bar{\ue}|^2 div V - 2\alpha \varepsilon \int_\Omega DV \nabla \bar{\ue} \cdot \nabla \bar{\ue} + \frac{\alpha}{\varepsilon} \int_{\Omega} \bar{\ue}(1-\bar{\ue})div V \rightarrow \int_\Omega(div V - DV \normal \cdot \normal)d|D\bar{u}|
\]
is proved in \cite{art:garcke2008}, Theorem 4.2 (see also annotations in \cite{art:garcke2016}, proof of Theorem 21).
\end{proof}
In particular, we point out that this implies, together with Proposition \ref{prop:conv2}, that the expression of the optimality condition for the phase field problem converges, as $\varepsilon \rightarrow 0$, to the one in the sharp case.

\subsection{Comparison with the shape derivative algorithm}
\label{shapeder}
In this section, we report some results of the application of the algorithm based on the shape derivative. In the implementation, we take advantage of the Finite Element method to solve the direct and adjoint problems and compute the shape gradient as in \eqref{eq:shapeder}. We consider an initial guess for the inclusion (in all the simulations reported, the initial guess is a disc centered in the origin with radius $0.02$) and discretize its boundary with a finite number of points, which always coincide with vertices of the numerical mesh. We iteratively perturb the inclusion by moving the boundary with a vector field $V$ which is the projection in the finite element space $V_h \times V_h$ of the shape gradient reported in \eqref{eq:shapeder} and $\normal_\omega$ the external normal vector of $\partial \omega$ (see e.g. \cite{art:marco} for more details). After the descent direction is determined, a backtracking scheme is implemented (see \cite{book:NocedalWright}), in order to guarantee the decrease of the cost functional $J$ at each iteration. As in the case of Algorithm \ref{al:POP}, we start from the initial guess $u^0 \equiv 0$ and take advantage of $N_f = 2$ measurements, associated to the same source terms. The main parameters of this set of simulations are reported in Table \ref{tab:shapeder}.
\begin{table}[h!]
\centering
\begin{tabular}{c|c|c}
$\alpha$ & $max step$ & $tol$ \\
\hline
$10^{-3}$ & $10$ & $10^{-6}$ 
\end{tabular}
\caption{Values of the main parameters}
\label{tab:shapeder}
\end{table}
\par  
In Figure \ref{fig:uzawa} we report the results of the reconstruction with the shape gradient algorithm compared to the ones of the Parabolic Obstacle problem (with $\varepsilon = \frac{1}{16\pi}$ and with mesh adaptation). Each result is endowed with a plot of the evolution of the cost functional throughout time (in particular, of $J_{PDE}(u) = \half \norm{L^2(\partial \Omega)}{S(u) - y_{meas}}$
\begin{figure}[h!]
			\centering
			\subfloat[Shape gradient, evolution of the cost functional]{
		    	\includegraphics[width=0.33\textwidth]{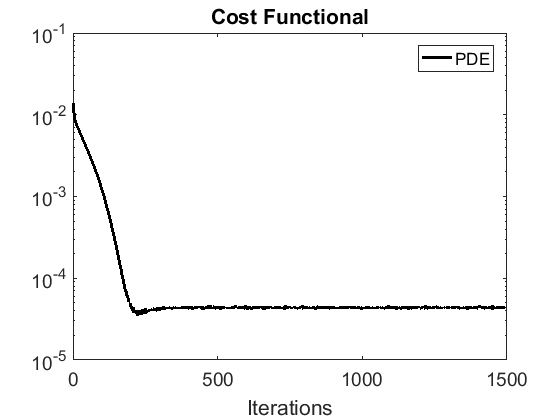}
			}
			\subfloat[Shape gradient, $N_{tot} = 1494$]{
		    	\includegraphics[width=0.33\textwidth]{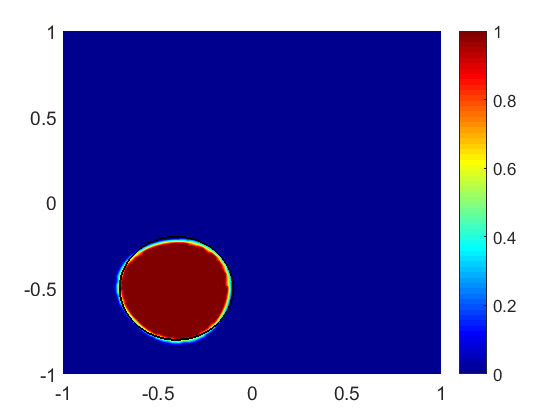}
			}
			\subfloat[Phase field, $\varepsilon = \frac{1}{16\pi}$, $N_{tot} = 1869$]{
		    	\includegraphics[width=0.33\textwidth]{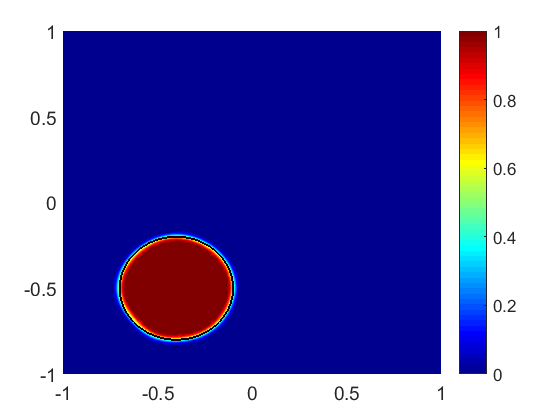}
			}
			\\
			\subfloat[Shape gradient, evolution of the cost functional]{
		    	\includegraphics[width=0.33\textwidth]{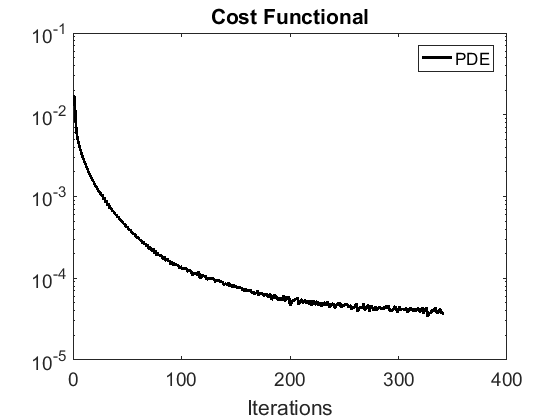}
			}
			\subfloat[Shape gradient, $N_{tot} = 301$]{
		    	\includegraphics[width=0.33\textwidth]{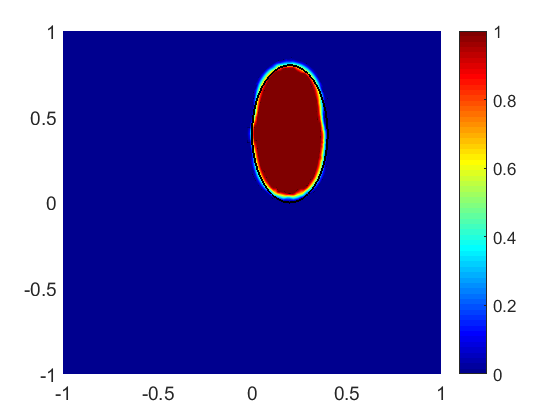}
			}
			\subfloat[Phase field, $\varepsilon = \frac{1}{16\pi}$, $N_{tot} = 1503$]{
		    	\includegraphics[width=0.33\textwidth]{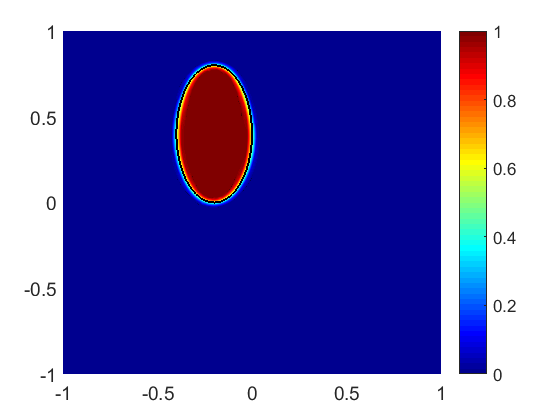}
			}
	\caption{Shape gradient algorithm: result comparison}
	\label{fig:uzawa}
\end{figure}
The reconstruction achieved by the shape gradient algorithm seems to be as accurate as the phase-field one. The sharp method seems to be less expensive in term of computational cost, and it involves a smaller number of iterations. Nevertheless, it requires the knowledge of remarkable \textit{a priori} information, i.e. the topology of the inclusion.

\section*{Appendix}
In the proof of various proposition we have used the following generalized Poincarè inequality:
\begin{lemma}
	Let $\Omega^* \subset \Omega$ be s.t. $|\Omega^*| \neq 0$. Then $\exists C > 0, C = C(\Omega)$ s.t., $\forall u \in H^1(\Omega)$,
	\begin{equation}
		\norm{H^1(\Omega)}{u}^2 \leq C \left( \norm{L^2(\Omega)}{\nabla u}^2 + \norm{L^2(\Omega^*)}{u}^2 \right).
	\label{eq:Poincare}
	\end{equation}
	\label{Lemma1}
\end{lemma} 
The proof of the Lemma is given in the Appendix of \cite{art:bcmp} and easily follows by Theorem 8.11 in \cite{liebloss}.
\par
Thanks to Lemma \ref{Lemma1}, we can prove the following well-posedness result for the direct problem. We remark that a similar analysis was performed in \cite{art:bcmp}, but here we extend the result for the case of inclusions which have the property of being finite-perimeter sets.
\begin{proposition}
	Consider $f \in \left(H^1(\Omega)\right)^*$ and a function $u \in BV(\Omega;[0,1])$ s.t. $u$ is not (a.e.) equal to $1$. Then there exists an unique solution $S(u)\in H^1(\Omega)$ of
	\[
	\int_\Omega a(u) \nabla S(u) \cdot \nabla v + \int_\Omega b(u)S(u)^3v = \int_\Omega fv \qquad \forall v \in H^1(\Omega),
	\]
	where $a(u) = 1-(1-k)u$ and $b(u) = 1-u$.
\label{prop:wellpos}
\end{proposition}  
\begin{proof}
We use the Minty-Browder theorem (see, e.g., Theorem 5.16 in \cite{book:brezis}), introducing (for a fixed $u$) the operator $T:H^1(\Omega) \rightarrow \left(H^1(\Omega)\right)^*$ s.t.
\[
	\langle T(S),v \rangle_* = \int_\Omega a(u)\nabla S \cdot \nabla v + \int_\Omega b(u) S^3 v.
\]
We can easily verify that the nonlinear operator $T$ is continuous, coercive and monotone.
\begin{itemize}
	\item Continuity: we indeed prove that $T$ is locally Lipschitz continuous with respect to $S$.
	\[
	\begin{aligned}
	|\langle T(S)-T(S_0), v \rangle_*| &= \left| \int_\Omega a(u) \nabla (S-S_0) \cdot \nabla v + \int_\Omega b(u)(S-S_0)q\right| \textit{ (being $q = S^2+S S_0 + S_0^2$)} \\
	& \leq \norm{L^2}{\nabla (S-S_0)}\norm{L^2}{\nabla v} + \norm{L^6}{S-S_0}\norm{L^3}{ q}\norm{L^2}{v}
	\end{aligned}
	\]
	If $S$ and $S_0$ belong to a bounded subset of $H^1(\Omega)$, then (thanks to the Sobolev Embedding of $H^1(\Omega)$ in $L^6(\Omega)$) we can assess that $\norm{L^3}{q}\leq M$ and moreover $\exists K>0$ s.t. 
	\[
	|\langle T(S)-T(S_0), v \rangle_*|\leq K \norm{H^1}{S-S_0} \norm{H^1}{v} \qquad \forall v \in H^1(\Omega).
	\]
	\item Coercivity: we show that $\langle T(S),S \rangle_* \rightarrow +\infty $ as $\norm{H^1(\Omega)}{S} \rightarrow + \infty$. Since $u$ is not identically equal to $1$, $\exists Q>0$ and $\Omega^*: |\Omega^*|\neq 0$ s.t. $b(u) = 1-u \geq Q$ a.e. in $\Omega^*$. Then,
	\[
	\begin{aligned}
	\langle T(S),S \rangle_* &\geq k \int_\Omega |\nabla S|^2 + Q\int_{\Omega^*}S^4 \geq k \norm{L^2(\Omega)}{\nabla S}^2 + \frac{Q}{|\Omega|}\norm{L^2(\Omega^*)}{S}^4 \\
	&=k \left( \norm{L^2(\Omega)}{\nabla S}^2 + \norm{L^2(\Omega^*)}{S}^2 \right) + R,	
	\end{aligned}
	\]
where $R = \frac{Q}{|\Omega|}\norm{L^2(\Omega^*)}{S}^4 - k \norm{L^2(\Omega^*)}{S}^2$ can be bounded by below independently of $S$ by considering that 
\[
\left( A\norm{L^2(\Omega^*)}{S}^2 - B \right)^2 \geq 0 \Rightarrow A^2 \norm{L^2(\Omega^*)}{S}^4 - 2AB \norm{L^2(\Omega^*)}{S}^2 \geq - B^2,
\]
which implies (chosen $A=\sqrt{\frac{Q}{|\Omega|}}$ and $B = \frac{k \sqrt{|\Omega|}}{2 \sqrt{Q}}$) that $R \geq -\frac{k^2 |\Omega|}{4Q}$. Together with the Poincarè inequality in Lemma \ref{Lemma1}, we conclude that
\[
\langle T(S),S \rangle_* \geq \frac{k}{C} \norm{H^1(\Omega)}{S}^2 - \frac{k^2 |\Omega|}{4Q}.
\]
\item (Strict) monotonicity: we claim that $\langle T(S) - T(R), S-R \rangle_* \geq 0$ and $\langle T(S) - T(R), S-R \rangle_* = 0 \Leftrightarrow S = R$. Indeed, 
\[
\langle T(S) - T(R), S-R \rangle_*  \geq \int_\Omega k |\nabla(S-R)|^2 + Q \int_{\Omega^*} (S^2+SR+R^2)(S-R)^2 \geq 0.
\] 
Moreover, since $S^2+SR+R^2 \geq \frac{1}{4}(S-R)^2$,
\[
\langle T(S) - T(R), S-R \rangle_* = 0 \Rightarrow \norm{L^2(\Omega)}{\nabla(S-R)}=0 \text{ and } \int_{\Omega^*} (S-R)^4 = 0,
\]
and from the latter equality it follows that $S=R$ a.e. in $\Omega^*$, hence also $\norm{L^2(\Omega^*)}{S-R}=0$, and thanks to the Poincaré inequality in Lemma \ref{Lemma1}, $\norm{H^1(\Omega)}{S-R}=0$.
\end{itemize}
\end{proof}
Finally, we prove an estimate which occurs many times in the proof of various results.
\begin{proposition}
Suppose that $f\in L^2(\Omega)$ s.t. $\int_\Omega f \neq 0$. Consider $S(u)$ the solution of problem \eqref{eq:prob} associated to $u \in BV(\Omega;[0,1])$, $u$ not identically equal to $1$. Then, there exists $\Omega^*$ and $Q>0$ s.t. $|\Omega^*|\neq 0$ and 
\[
b(u)S(u)^2 \geq Q \qquad \text{ a.e. in }\Omega^* 
\] 
\label{prop:geq}
\end{proposition} 
\begin{proof}
By contraddiction, suppose the opposite of the thesis: $b(u)S(u)^2 = 0$ a.e. in $\Omega$. Then, this would imply that $\int_\Omega b(u) S(u)^3 = 0$, and then it would hold that
\[
\int_\Omega a(u)\nabla S(u)\cdot \nabla v = \int_\Omega f v \qquad \forall v \in H^1(\Omega).
\] 
Taking $v = const.$ we obtain that $\int_\Omega f = 0$, which contraddicts the hypothesis.
\end{proof}

We remark that the previous result can be extended to class of functions $f$ satisfying more general hypotheses. If, for example, we restrict to the case of inclusions well separated from the boundary ($u = 0$ a.e. in $\Omega^{d_0}$, being $\Omega^{d_0} = \{x \in \Omega: dist(x,\partial \Omega) \leq d_0\}$, $d_0 > 0$), then it is sufficient to require $f \neq 0$ a.e in $\Omega^{d_0}$ to guarantee an estimate equivalent to the one of Proposition \ref{prop:geq}.
\begin{proposition}
Suppose that $u \in BV(\Omega;[0,1])$ satisfies $u = 0$ a.e. in $\Omega^{d_0}$. If $f\in L^2(\Omega)$ does not vanish in $\Omega^{d_0}$then there exists $Q>0$ s.t. the solution $S(u)$ of \eqref{eq:prob} satisfies
\[
b(u)S(u)^2 \geq Q \qquad \text{ a.e. in }\Omega^{d_0} 
\] 
\label{prop:geq2}
\end{proposition}
\begin{proof}
By contraddiction of the thesis, suppose $S(u) \equiv 0$ in $\Omega^{d_0}$ and recall $\Omega^{in} = \Omega \setminus \Omega^{d_0}$; then it holds
\begin{equation}
\int_{\Omega^{in}}a(u) \nabla S(u)\cdot \nabla v + \int_{\Omega^{in}} b(u) S(u)^3 v = \int_{\Omega} fv \quad \forall v \in H^1(\Omega)
\label{eq:geq2eq}
\end{equation}
	The space $H^1_0(\Omega^{d_0})$, obtained by closing the space of all the smooth function whose support is compactly contained in $\Omega^{d_0}$ with respect to the $H^1$ norm, is well defined; moreover $H^1_0(\Omega^{d_0})\subset H^1(\Omega)$. Hence, equation \eqref{eq:geq2eq} holds for all $v \in H^1_0(\Omega^{d_0})$, and this implies that 
	\[
	\int_\Omega f v = 0 \qquad \forall v \in H^1_0(\Omega^{d_0}),
	\]
which eventually entails that $f=0$ a.e. in $\Omega^{d_0}$, that is a contraddiction with hypotheses.
	\end{proof}

\section*{Acknowledgments}
E. Beretta and M. Verani thank the New York University in Abu Dhabi for its kind hospitality that permitted a further development of the present research. We acknowledge the use of the MATLAB library redbKIT \cite{redbKIT} for the numerical simulations presented in this work.

\printbibliography
\end{document}